\documentclass{article}
\usepackage{amsmath,amssymb,amsthm,dsfont,color}
\usepackage{amsmath,amssymb,amsthm,mathrsfs,amsfonts,dsfont}

\newtheorem{lemma}{Lemma}[section]

\newtheorem{claim}{Claim}[section]

\begin{document}
\thispagestyle{empty}

{\noindent \LARGE Mappings preserving products $ab\pm ba^{*}$ on alternative $W^{*}$-factors}
\vspace{.2in}
\begin{center}
{\bf Jo\~{a}o Carlos da Motta Ferreira\\ and \\Maria das Gra\c{c}as Bruno Marietto}\\
\vspace{.2in}
{\it Center for Mathematics, Computation and Cognition,\\
Federal University of ABC,\\
Avenida dos Estados, 5001, \\
09210-580, Santo Andr\'e, Brazil.}\\
e-mail: joao.cmferreira@ufabc.edu.br, graca.marietto@ufabc.edu.br\\
\end{center}

\begin{abstract} Let $\mathcal{A}$ and $\mathcal{B}$ be two alternative $W^{*}$-factors. In this paper, we proved that a bijective mapping $\Phi :\mathcal{A}\rightarrow \mathcal{B}$ satisfies $\Phi (ab+ba^{*})=\Phi (a)\Phi (b)+\Phi (b)\Phi (a)^{*}$ (resp., $\Phi (ab-ba^{*})=\Phi (a)\Phi (b)-\Phi (b)\Phi (a)^{*}$), for all elements $a,b\in \mathcal{A}$, if and only if $\Phi $ is a $\ast $-ring isomorphism.
\end{abstract}
{\bf Abbreviated title:} $\ast $-ring isomorphisms on alternative $W^{*}$-factors

\section{Introduction}

A ring $\mathcal{R}$ not necessarily associative or commutative is called an {\it alternative ring} if it satisfies the identities $a^{2}b=a(ab)$ and $ba^{2}=(ba)a,$ for all elements $a,b\in \mathcal{R}.$ One easily sees that any associative ring is an alternative ring. An alternative ring $\mathcal{R}$ is called {\it prime} if for any elements $a,b\in \mathcal{R}$ satisfying the condition $a\mathcal{R}b=0,$ then either $a=0$ or $b=0.$

Let $\mathcal{R}$ and $\mathcal{S}$ be alternative rings. We say that a mapping $\Phi:\mathcal{R}\rightarrow \mathcal{S}$ {\it preserves product} if $\Phi (ab)=\Phi (a)\Phi (b),$ for all elements $a,b\in \mathcal{R},$ and that it {\it preserves Jordan product} (resp., {\it preserves Lie product}) if $\Phi (ab+ba)=\Phi (a)\Phi (b)+\Phi (b)\Phi (a)$ (resp., $\Phi (ab-ba)=\Phi (a)\Phi (b)-\Phi (b)\Phi (a)$), for all elements $a,b\in \mathcal{R}.$ 

We say that a mapping $\Phi:\mathcal{R}\rightarrow \mathcal{S}$ is {\it additive} if $\Phi (a+b)=\Phi (a)+\Phi (b),$ for all elements $a,b\in \mathcal{R}$ and that it is a {\it ring isomorphism} if $\Phi $ is an additive bijection that preserves products.

Let $\mathcal{R}$ and $\mathcal{S}$ be alternative $\ast $-rings. We say that a mapping $\Phi:\mathcal{R}\rightarrow \mathcal{S}$ {\it preserves involution} if $\Phi (a^{*})=\Phi (a)^{*},$ for all elements $a\in \mathcal{R},$ and that $\Phi $ is a {\it $\ast $-ring isomorphism} if $\Phi $ is a ring isomorphism that preserves involution.

There has been a great interest in the study of additivity or characterization of isomorphisms of mappings that preserve product or Jordan product (resp., Lie product) on associative and non-associative rings as well as operator algebras (for example, see \cite{BrunGuz}, \cite{FerGuz1}, \cite{FerGuz2}, \cite{Ji}, \cite{Ji2}, \cite{Lu1}, \cite{Lu2}, \cite{Martindale}, \cite{Martindale1}, \cite{Miers1}, \cite{Miers2} and \cite{Zhang}). As in the case of associative rings, many mathematicians devoted themselves to study mappings preserving new products on operator algebras with involution (for example, see \cite{Bai}, \cite{Cui}, \cite{FerMar} and \cite{Li}). In particular, Li et al. \cite{Li} (resp., Cui and Li \cite{Cui}) studied bijective mappings preserving the new product $ab+ba^{*}$ (resp., $ab-ba^{*}$). They showed that such mappings on factor von Neumann algebras are $\ast $-ring isomorphisms.

A complete normed alternative complex $\ast $-algebra $\mathcal{A}$ is called of {\it alternative $C^{*}$-algebra} if it satisfies the condition: $\|a^{*}a\|=\|a\|^{2},$ for all elements $a\in \mathcal{A}.$ A non-zero element $p\in \mathcal{A}$ is called a {\it projection} if it is self-adjoint and verifies the condition $p^{2}=p.$ Alternative $C^{*}$-algebras are non-associative generalizations of $C^{*}$-algebras and appear in various areas in Mathematics (see more details in the references \cite{Cabrera1} and \cite{Cabrera2}). An alternative $C^{*}$-algebra $\mathcal{A}$ is called of {\it alternative $W^{*}$-algebra} if it is a dual Banach space and a prime alternative $W^{*}$-algebra is called {\it alternative $W^{*}$-factor}. It is well known that non-zero alternative $W^{*}$-algebras are unital and that an alternative $W^{*}$-algebra $\mathcal{A}$ is a factor if and only if its center is equal to $\mathds{C}1_{\mathcal{A}},$ where $1_{\mathcal{A}}$ is the unit of $\mathcal{A}.$ 

Similarly to the research performed in \cite{Li} (resp., \cite{Cui}), the aim of the present paper is to investigate when a bijective mapping preserving product $ab+ba^{*}$ (resp., $ab-ba^{*}$) on an alternative $W^{*}$-factor is a $\ast $-ring isomorphism.

Let $\mathcal{A}$ and $\mathcal{B}$ be two alternative complex $\ast $-algebras. We say that a mapping $\Phi:\mathcal{A}\rightarrow \mathcal{B}$ {\it preserves product $ab+ba^{*}$} (resp., {\it preserves product $ab-ba^{*}$}) if
{\allowdisplaybreaks\begin{multline}\allowdisplaybreaks\label{fundident}
\Phi (ab+ba^{*})=\Phi (a)\Phi (b)+\Phi (b)\Phi (a)^{*}\\ (\textrm{resp.},\, \Phi (ab-ba^{*})=\Phi (a)\Phi (b)-\Phi (b)\Phi (a)^{*}),
\end{multline}}
for all elements $a,b\in \mathcal{A}.$

Our main result reads as follows.

\noindent {\bf Main Theorem.} {\it Let $\mathcal{A}$ and $\mathcal{B}$ be two alternative $W^{*}$-factors with $1_{\mathcal{A}}$ and $1_{\mathcal{B}}$ the identities of them. Then a bijective mapping $\Phi:\mathcal{A}\rightarrow \mathcal{B}$ satisfies $\Phi (ab+ba^{*})=\Phi (a)\Phi (b)+\Phi (b)\Phi (a)^{*}$ (resp., $\Phi (ab-ba^{*})=\Phi (a)\Phi (b)-\Phi (b)\Phi (a)^{*}$), for all elements $a,b\in \mathcal{ A},$ if and only if $\Phi $ is a $\ast $-ring isomorphism}.

\section{The proof of the Main Theorem}

The proof of the Main Theorem is made by proving several lemmas and taking into account separately the products $ab+ba^{*}$ and $ab-ba^{*}.$ We begin with the following lemma.

\begin{lemma}\label{lem21} Let $\mathcal{A}$ and $\mathcal{B}$ be two alternative complex $\ast $-algebras and $\Phi:\mathcal{A}\rightarrow \mathcal{B}$ a mapping that preserves product $ab+ba^{*}$ (resp., $ab-ba^{*}$). If $\Phi (f)=\Phi (a)+\Phi (b),$ for elements $a,b,f$ of $\mathcal{A},$ then hold the following identities: (i) $\Phi (ft+tf^{*})=\Phi (at+ta^{*})+\Phi (bt+tb^{*})$ and (ii) $\Phi (tf+ft^{*})=\Phi (ta+at^{*})+\Phi (tb+bt^{*})$ (resp., (i) $\Phi (ft-tf^{*})=\Phi (at-ta^{*})+\Phi (bt-tb^{*})$ and (ii) $\Phi (tf-ft^{*})=\Phi (ta-at^{*})+\Phi (tb-bt^{*})$), for all elements $t$ of $\mathcal{A}.$
\end{lemma}
\begin{proof} First case: the product $ab+ba^{*}.$ For an arbitrary element $t$ of $\mathcal{A},$ we have
{\allowdisplaybreaks\begin{multline*}\allowdisplaybreaks
\Phi (ft+tf^{*})=\Phi (f)\Phi (t)+\Phi (t)\Phi (f)^{*}\\=(\Phi (a)+\Phi (b))\Phi (t)+\Phi (t)(\Phi (a)+\Phi (b))^{*}\\
=\Phi (a)\Phi (t)+\Phi (t)\Phi (a)^{*}+\Phi (b)\Phi (t)+\Phi (t)\Phi (b)^{*}\\
=\Phi (at+ta^{*})+\Phi (bt+tb^{*}).
\end{multline*}}
Similarly, we prove (ii).

\noindent Second case: the product $ab-ba^{*}.$ The proofs are similar to those in previous.

\end{proof}

\begin{lemma}\label{lem22} Let $\mathcal{A}$ and $\mathcal{B}$ be two alternative complex $\ast $-algebras and $\Phi:\mathcal{A}\rightarrow \mathcal{B}$ a surjective mapping that preserves product $ab+ba^{*}$ (resp., $ab-ba^{*}$). Then $\Phi (0)=0.$
\end{lemma}
\begin{proof} First case: the product $ab+ba^{*}.$ From the surjectivity of $\Phi ,$ choose an element $b\in \mathcal{A}$ such that $\Phi (b)=0.$ It follows that
{\allowdisplaybreaks\begin{eqnarray*}\allowdisplaybreaks
&\Phi (0)=\Phi (0b+b0^{*})=\Phi (0)\Phi (b)+\Phi (b)\Phi (0)^{*}=\Phi (0)0+0\Phi (0)^{*}=0.
\end{eqnarray*}}

\noindent Second case: the product $ab-ba^{*}.$ The proof is similar to the previous one. 
\end{proof}

\begin{lemma}\label{lem23} Let $\mathcal{A}$ and $\mathcal{B}$ be two alternative complex $\ast $-algebras such that $\mathcal{A}$ is an alternative $W^{*}$-factor with identity $1_{\mathcal{A}}.$ Then every bijective mapping $\Phi:\mathcal{A}\rightarrow \mathcal{B}$ that preserves product $ab+ba^{*}$ (resp., $ab-ba^{*}$) is additive.
\end{lemma}

We will establish the proof of Lemma \ref{lem23} in a series of Claims and taking into account separately the products $ab+ba^{*}$ and $ab-ba^{*},$ based on the techniques used by Ferreira and Marietto \cite{FerMar} and Li et al. \cite{Li} (resp., Cui and Li \cite{Cui}). We begin, though, with a well-known result that will be used throughout this paper.

Let $p_{1}$ be a non-trivial projection of $\mathcal{A}$ and write $p_{2}=1_{\mathcal{A}}-p_{1}.$ Then $\mathcal{A}$ has a Peirce decomposition $\mathcal{A}=\mathcal{A}_{11}\oplus \mathcal{A}_{12}\oplus \mathcal{A}_{21}\oplus \mathcal{A}_{22},$ where $\mathcal{A}_{ij}=p_{i}\mathcal{A}p_{j}$ $(i,j=1,2) ,$ satisfying the following multiplicative relations: (i) $\mathcal{A}_{ij}\mathcal{A}_{jl}\subseteq \mathcal{A}_{il}$ $(i,j,l=1,2),$ (ii) $\mathcal{A}_{ij}\mathcal{A}_{ij}\subseteq \mathcal{A}_{ji}$ $(i,j=1,2),$ (iii) $\mathcal{A}_{ij}\mathcal{A}_{kl}=0$ if $j\neq k$ and $(i,j)\neq (k,l)$ $(i,j,k,l=1,2)$ and (iv) $a_{ij}^{2}=0$ for all elements $a_{ij}\in \mathcal{A}_{ij}$ $(i,j=1,2;i\neq j).$ Moreover, it follows from \cite[Theorem 2.2]{FerGuz1} that if $a_{ij}t_{jk}=0$ for each $t_{jk}\in \mathcal{A}_{jk}$ $(i,j,k=1,2),$ then $a_{ij}=0.$ Dually, if $t_{ki}a_{ij}=0$ for each $t_{ki}\in \mathcal{A}_{ki}$ $(i,j,k=1,2),$ then $a_{ij}=0.$

\begin{claim}\label{claim21} For arbitrary elements $a_{11}\in \mathcal{A}_{11},$ $b_{12}\in \mathcal{A}_{12},$ $c_{21}\in \mathcal{A}_{21}$ and $d_{22}\in \mathcal{A}_{22}$ hold: (i) $\Phi (a_{11}+b_{12})=\Phi (a_{11})+ \Phi (b_{12}),$ (ii) $\Phi (a_{11}+c_{21})=\Phi (a_{11})+ \Phi (c_{21}),$ (iii) $\Phi (b_{12}+d_{22})=\Phi (b_{12})+\Phi (d_{22})$ and (iv) $\Phi (c_{21}+d_{22})=\Phi (c_{21})+\Phi (d_{22}).$
\end{claim}
\begin{proof} First case: the product $ab+ba^{*}.$ From the surjectivity of $\Phi $ there exists an element $f=f_{11}+f_{12}+f_{21}+f_{22}\in \mathcal{A}$ such that $\Phi (f)=\Phi (a_{11}) + \Phi (b_{12})$. By Lemma \ref{lem21}(ii), we have
{\allowdisplaybreaks\begin{eqnarray*}\allowdisplaybreaks
&\Phi (p_{2}f+fp_{2}^{*})=\Phi (p_{2}a_{11}+a_{11}p_{2}^{*})+\Phi (p_{2}b_{12}+b_{12}p_{2}^{*})=\Phi (b_{12}).
\end{eqnarray*}}
This implies that $p_{2}f+fp_{2}^{*}=b_{12}$ which results that $f_{12}=b_{12},$ $f_{21}=0$ and $f_{22}=0,$ by directness of the Peirce decomposition. It follows that $\Phi (f_{11}+b_{12})=\Phi (a_{11})+\Phi (b_{12}).$ Hence, for an arbitrary element $d_{21}\in  \mathcal{A}_{21}$ we have
{\allowdisplaybreaks\begin{eqnarray*}\allowdisplaybreaks
&\Phi (d_{21}(f_{11}+b_{12})+(f_{11}+b_{12})d_{21}^{*})=\Phi (d_{21}a_{11}+a_{11}d_{21}^{*})+\Phi (d_{21}b_{12}+b_{12}d_{21}^{*})
\end{eqnarray*}}
which yields that 
{\allowdisplaybreaks\begin{multline*}\allowdisplaybreaks
\Phi ((d_{21}(f_{11}+b_{12})+(f_{11}+b_{12})d_{21}^{*})p_{2}+p_{2}(d_{21}(f_{11}+b_{12})+(f_{11}+b_{12})d_{21}^{*})^{*})\\
=\Phi ((d_{21}a_{11}+a_{11}d_{21}^{*})p_{2}+p_{2}(d_{21}a_{11}+a_{11}d_{21}^{*})^{*})\\
+\Phi ((d_{21}b_{12}+b_{12}d_{21}^{*})p_{2}+
p_{2}(d_{21}b_{12}+b_{12}d_{21}^{*})^{*}).
\end{multline*}}
As consequence, we obtain
{\allowdisplaybreaks\begin{multline*}\allowdisplaybreaks
\Phi (d_{21}b_{12}+f_{11}d_{21}^{*}+b_{12}^{*}d_{21}^{*}+d_{21}f_{11}^{*})\\
=\Phi (a_{11}d_{21}^{*}+d_{21}a_{11}^{*})+\Phi (d_{21}b_{12}+b_{12}^{*}d_{21}^{*}).
\end{multline*}}
It follows from this that
{\allowdisplaybreaks\begin{multline*}\allowdisplaybreaks
\Phi (p_{1}(d_{21}b_{12}+f_{11}d_{21}^{*}+b_{12}^{*}d_{21}^{*}+d_{21}f_{11}^{*})+(d_{21}b_{12}+f_{11}d_{21}^{*}+b_{12}^{*}d_{21}^{*}+d_{21}f_{11}^{*})p_{1}^{*})\\
=\Phi (p_{1}(a_{11}d_{21}^{*}+d_{21}a_{11}^{*})+(a_{11}d_{21}^{*}+d_{21}a_{11}^{*})p_{1}^{*})\\+\Phi (p_{1}(d_{21}b_{12}+b_{12}^{*}d_{21}^{*})+(d_{21}b_{12}+b_{12}^{*}d_{21}^{*})p_{1}^{*})
\end{multline*}}
which implies that
{\allowdisplaybreaks\begin{eqnarray*}\allowdisplaybreaks
\Phi (f_{11}d_{21}^{*}+d_{21}f_{11}^{*})=\Phi (a_{11}d_{21}^{*}+d_{21}a_{11}^{*}).
\end{eqnarray*}}
From the injectivity of $\Phi ,$ we obtain $f_{11}d_{21}^{*}+d_{21}f_{11}^{*}=a_{11}d_{21}^{*}+d_{21}a_{11}^{*}$ which show that $f_{11}=a_{11}.$ Now, choose $f=f_{11}+f_{12}+f_{21}+f_{22}\in \mathcal{A}$ such that $\Phi (f)=\Phi (a_{11}) + \Phi (c_{21})$. By Lemma \ref{lem21}(ii) again, we have
{\allowdisplaybreaks\begin{eqnarray*}\allowdisplaybreaks
&\Phi (p_{2}f+fp_{2}^{*})=\Phi (p_{2}a_{11}+a_{11}p_{2}^{*})+\Phi (p_{2}c_{21}+c_{21}p_{2}^{*})=\Phi (c_{21}).
\end{eqnarray*}}
which results that $p_{2}f+fp_{2}^{*}=c_{21}.$ From this we conclude that $f_{12}=0,$ $f_{21}=c_{21}$ and $f_{22}=0.$ As consequence, we obtain $\Phi (f_{11}+c_{21})=\Phi (a_{11})+\Phi (c_{21}).$ Hence, for an arbitrary element $d_{21}\in  \mathcal{A}_{21}$ we have
{\allowdisplaybreaks\begin{eqnarray*}\allowdisplaybreaks
&\Phi (d_{21}(f_{11}+c_{21})+(f_{11}+c_{21})d_{21}^{*})=\Phi (d_{21}a_{11}+a_{11}d_{21}^{*})+\Phi (d_{21}c_{21}+c_{21}d_{21}^{*})
\end{eqnarray*}}
which leads to 
{\allowdisplaybreaks\begin{multline*}\allowdisplaybreaks
\Phi ((d_{21}(f_{11}+c_{21})+(f_{11}+c_{21})d_{21}^{*})p_{1}+p_{1}(d_{21}(f_{11}+c_{21})+(f_{11}+c_{21})d_{21}^{*})^{*})\\
=\Phi ((d_{21}a_{11}+a_{11}d_{21}^{*})p_{1}+p_{1}(d_{21}a_{11}+a_{11}d_{21}^{*})^{*})\\
+\Phi ((d_{21}c_{21}+c_{21}d_{21}^{*})p_{1}+p_{1}(d_{21}c_{21}+c_{21}d_{21}^{*})^{*})
\end{multline*}}
resulting in
{\allowdisplaybreaks\begin{eqnarray*}\allowdisplaybreaks
\Phi (d_{21}f_{11}+f_{11}^{*}d_{21}^{*})=\Phi (d_{21}a_{11}+a_{11}^{*}d_{21}^{*}).
\end{eqnarray*}}
From the injectivity of $\Phi ,$ we obtain $d_{21}f_{11}+f_{11}^{*}d_{21}^{*}=d_{21}a_{11}+a_{11}^{*}d_{21}^{*}$ which yields in $f_{11}=a_{11}.$

Similarly, we prove the cases (iii) and (iv).

\noindent Second case: the product $ab-ba^{*}.$ According to the hypothesis on $\Phi ,$ choose $f=f_{11}+f_{12}+f_{21}+f_{22}\in \mathcal{A}$ such that $\Phi (f)=\Phi (a_{11}) + \Phi (b_{12})$. By Lemma \ref{lem21}(ii), we have
{\allowdisplaybreaks\begin{eqnarray*}\allowdisplaybreaks
&\Phi (p_{1}f-fp_{1}^{*})=\Phi (p_{1}a_{11}-a_{11}p_{1}^{*})+\Phi (p_{1}b_{12}-b_{12}p_{1}^{*})=\Phi (b_{12}).
\end{eqnarray*}}
This implies that $p_{1}f-fp_{1}^{*}=b_{12}$ which results that $f_{12}=b_{12}$ and $f_{21}=0,$ by directness of the Peirce decomposition. We have thus shown that $\Phi (f_{11}+b_{12}+f_{22})=\Phi (a_{11})+\Phi (b_{12}).$ Next, for any element $d_{21}\in  \mathcal{A}_{12}$ we have
{\allowdisplaybreaks\begin{multline*}\allowdisplaybreaks
\Phi (d_{21}(f_{11}+b_{12}+f_{22})-(f_{11}+b_{12}+f_{22})d_{21}^{*})=\Phi (d_{21}a_{11}-a_{11}d_{21}^{*})\\+\Phi (d_{21}b_{12}-b_{12}d_{21}^{*})
\end{multline*}}
from which we get
{\allowdisplaybreaks\begin{multline*}\allowdisplaybreaks
\Phi (d_{21}f_{11}+d_{21}b_{12}-f_{11}d_{21}^{*}-b_{12}d_{21}^{*})=\Phi (d_{21}a_{11}-a_{11}d_{21}^{*})\\+\Phi (d_{21}b_{12}-b_{12}d_{21}^{*}).
\end{multline*}}
It follows from this that
{\allowdisplaybreaks\begin{multline*}\allowdisplaybreaks
\Phi ((d_{21}f_{11}+d_{21}b_{12}-f_{11}d_{21}^{*}-b_{12}d_{21}^{*})p_{2}-p_{2}(d_{21}f_{11}+d_{21}b_{12}-f_{11}d_{21}^{*}-b_{12}d_{21}^{*})^{*})\\=\Phi ((d_{21}a_{11}-a_{11}d_{21}^{*})p_{2}-p_{2}(d_{21}a_{11}-a_{11}d_{21}^{*})^{*})\\+\Phi ((d_{21}b_{12}-b_{12}d_{21}^{*})p_{2}-p_{2}(d_{21}b_{12}-b_{12}d_{21}^{*})^{*})
\end{multline*}}
which implies that
{\allowdisplaybreaks\begin{multline*}\allowdisplaybreaks
\Phi (d_{21}b_{12}-f_{11}d_{21}^{*}-b_{12}^{*}d_{21}^{*}+d_{21}f_{11}^{*})=\Phi (-a_{11}d_{21}^{*}+d_{21}a_{11}^{*})\\+\Phi (d_{21}b_{12}-b_{12}^{*}d_{21}^{*}).
\end{multline*}}
As a consequence, we obtain
{\allowdisplaybreaks\begin{multline*}\allowdisplaybreaks
\Phi ((d_{21}b_{12}-f_{11}d_{21}^{*}-b_{12}^{*}d_{21}^{*}+d_{21}f_{11}^{*})p_{1}-p_{1}(d_{21}b_{12}-f_{11}d_{21}^{*}-b_{12}^{*}d_{21}^{*}+d_{21}f_{11}^{*})^{*})\\=\Phi ((-a_{11}d_{21}^{*}+d_{21}a_{11}^{*})p_{1}-p_{1}(-a_{11}d_{21}^{*}+d_{21}a_{11}^{*})^{*})\\+\Phi ((d_{21}b_{12}-b_{12}^{*}d_{21}^{*})p_{1}-p_{1}(d_{21}b_{12}-b_{12}^{*}d_{21}^{*})^{*})
\end{multline*}}
from which we have that
{\allowdisplaybreaks\begin{eqnarray*}\allowdisplaybreaks
\Phi (d_{21}f_{11}^{*}-f_{11}d_{21}^{*})=\Phi (d_{21}a_{11}^{*}-a_{11}d_{21}^{*}).
\end{eqnarray*}} 
This show that $d_{21}f_{11}^{*}-f_{11}d_{21}^{*}=d_{21}a_{11}^{*}-a_{11}d_{21}^{*}$ which yields that $d_{21}f_{11}^{*}=d_{21}a_{11}^{*}$ and $f_{11}d_{21}^{*}=a_{11}d_{21}^{*}.$  As consequence, we obtain $f_{11}=a_{11}.$ We have thus shown that $\Phi (a_{11}+b_{12}+f_{22})=\Phi (a_{11})+\Phi (b_{12}).$ Hence, for any element $d_{12}\in  \mathcal{A}_{12}$ we have
{\allowdisplaybreaks\begin{multline*}\allowdisplaybreaks
\Phi (d_{12}(a_{11}+b_{12}+f_{22})-(a_{11}+b_{12}+f_{22})d_{12}^{*})=\Phi (d_{12}a_{11}-a_{11}d_{12}^{*})\\+\Phi (d_{12}b_{12}-b_{12}d_{12}^{*})=\Phi (d_{12}b_{12}-b_{12}d_{12}^{*})
\end{multline*}}
which implies that $d_{12}(a_{11}+b_{12}+f_{22})-(a_{11}+b_{12}+f_{22})d_{12}^{*}=d_{12}b_{12}-b_{12}d_{12}^{*}.$ It follows that $d_{12}f_{22}-f_{22}d_{12}^{*}=0$ which results that $d_{12}f_{22}=0$ and $f_{22}d_{12}^{*}=0.$ As consequence, we obtain $f_{22}=0.$ Now, choose $f=f_{11}+f_{12}+f_{21}+f_{22}\in \mathcal{A}$ such that $\Phi (f)=\Phi (a_{11}) + \Phi (c_{21})$. By Lemma \ref{lem21}(ii), we have
{\allowdisplaybreaks\begin{eqnarray*}\allowdisplaybreaks
&\Phi (p_{2}f-fp_{2}^{*})=\Phi (p_{2}a_{11}-a_{11}p_{2}^{*})+\Phi (p_{2}c_{21}-c_{21}p_{2}^{*})=\Phi (c_{21}).
\end{eqnarray*}}
This implies that $p_{2}f-fp_{2}^{*}=c_{21}$ which results that $f_{12}=0$ and $f_{21}=c_{21},$ by directness of the Peirce decomposition. This show that, $\Phi (f_{11}+c_{21}+f_{22})=\Phi (a_{11})+\Phi (c_{21}).$ It follows that, for any element $d_{21}\in  \mathcal{A}_{21}$ we have
{\allowdisplaybreaks\begin{multline*}\allowdisplaybreaks
\Phi (d_{21}(f_{11}+c_{21}+f_{22})-(f_{11}+c_{21}+f_{22})d_{21}^{*})=\Phi (d_{21}a_{11}-a_{11}d_{21}^{*})\\+\Phi (d_{21}c_{21}-c_{21}d_{21}^{*})
\end{multline*}}
which leads to
{\allowdisplaybreaks\begin{multline*}\allowdisplaybreaks
\Phi (d_{21}f_{11}+d_{21}c_{21}-f_{11}d_{21}^{*}-c_{21}d_{21}^{*})=\Phi (d_{21}a_{11}-a_{11}d_{21}^{*})\\+\Phi (d_{21}c_{21}-c_{21}d_{21}^{*}).
\end{multline*}}
Hence,
{\allowdisplaybreaks\begin{multline*}\allowdisplaybreaks
\Phi ((d_{21}f_{11}+d_{21}c_{21}-f_{11}d_{21}^{*}-c_{21}d_{21}^{*})p_{1}-p_{1}(d_{21}f_{11}+d_{21}c_{21}-f_{11}d_{21}^{*}-c_{21}d_{21}^{*})^{*})\\=\Phi ((d_{21}a_{11}-a_{11}d_{21}^{*})p_{1}-p_{1}(d_{21}a_{11}-a_{11}d_{21}^{*})^{*})\\+\Phi ((d_{21}c_{21}-c_{21}d_{21}^{*})p_{1}-p_{1}(d_{21}c_{21}-c_{21}d_{21}^{*})^{*})
\end{multline*}}
from which we get
{\allowdisplaybreaks\begin{eqnarray*}\allowdisplaybreaks
\Phi (d_{21}f_{11}-f_{11}^{*}d_{21}^{*})=\Phi (d_{21}a_{11}-a_{11}^{*}d_{21}^{*}).
\end{eqnarray*}}
This show that
$d_{21}f_{11}-f_{11}^{*}d_{21}^{*}=d_{21}a_{11}-a_{11}^{*}d_{21}^{*}$ which results that $d_{21}f_{11}=d_{21}a_{11}$ and $f_{11}^{*}d_{21}^{*}=a_{11}^{*}d_{21}^{*}.$ Thus $f_{11}=a_{11}.$ We have thus shown that $\Phi (a_{11}+c_{21}+f_{22})=\Phi (a_{11})+\Phi (b_{12}).$ Thereby, for any element $d_{12}\in  \mathcal{A}_{12}$ we have
{\allowdisplaybreaks\begin{multline*}\allowdisplaybreaks
\Phi (d_{12}(a_{11}+c_{21}+f_{22})-(a_{11}+c_{21}+f_{22})d_{12}^{*})=\Phi (d_{12}a_{11}-a_{11}d_{12}^{*})\\
+\Phi (d_{12}c_{21}-c_{21}d_{12}^{*})=\Phi (d_{12}c_{21}-c_{21}d_{12}^{*})
\end{multline*}}
which yields that $d_{12}(a_{11}+c_{21}+f_{22})-(a_{11}+c_{21}+f_{22})d_{12}^{*}=d_{12}c_{21}-c_{21}d_{12}^{*}.$ It follows that $d_{12}f_{22}-f_{22}d_{12}^{*}=0$ which implies that $d_{12}f_{22}=0$ and $f_{22}d_{12}^{*}=0.$ As consequence, we obtain $f_{22}=0.$

Similarly, we prove the cases (iii) and (iv).
\end{proof}

\begin{claim}\label{cliam22} For arbitrary elements $b_{12}\in \mathcal{A}_{12}$ and $c_{21}\in \mathcal{A}_{21}$ holds $\Phi (b_{12}+c_{21})=\Phi (b_{12})+ \Phi (c_{21}).$
\end{claim}
\begin{proof} First case: the product $ab+ba^{*}.$ Choose an element $f=f_{11}+f_{12}+f_{21}+f_{22}\in \mathcal{A}$ such that $\Phi (f)=\Phi (b_{12})+\Phi (c_{21}).$ Then 
{\allowdisplaybreaks\begin{multline*}\allowdisplaybreaks
\Phi (p_{1}f+fp_{1}^{*})=\Phi (p_{1}b_{12}+b_{12}p_{1}^{*})+\Phi (p_{1}c_{21}+c_{21}p_{1}^{*})\\
=\Phi (b_{12})+\Phi (c_{21})=\Phi (f).
\end{multline*}}
This implies that $p_{1}f+fp_{1}^{*}=f$ which results in $f_{11}=0$ and $f_{22}=0.$ It follows that $\Phi (f_{12}+f_{21})=\Phi (b_{12})+\Phi (c_{21}).$ Hence, for an arbitrary element $d_{21}\in \mathcal{A}_{21}$ we have 
{\allowdisplaybreaks\begin{eqnarray*}\allowdisplaybreaks
&\Phi (d_{21}(f_{12}+f_{21})+(f_{12}+f_{21})d_{21}^{*})=\Phi (d_{21}b_{12}+b_{12}d_{21}^{*})+\Phi (d_{21}c_{21}+c_{21}d_{21}^{*})
\end{eqnarray*}}
which leads to
{\allowdisplaybreaks\begin{multline*}\allowdisplaybreaks
\Phi ((d_{21}(f_{12}+f_{21})+(f_{12}+f_{21})d_{21}^{*})p_{1}+p_{1}(d_{21}(f_{12}+f_{21})+(f_{12}+f_{21})d_{21}^{*})^{*})\\=\Phi ((d_{21}b_{12}+b_{12}d_{21}^{*})p_{1}+p_{1}(d_{21}b_{12}+b_{12}d_{21}^{*})^{*})\\+\Phi ((d_{21}c_{21}+c_{21}d_{21}^{*})p_{1}+p_{1}(d_{21}c_{21}+d_{21}^{*})^{*})
\end{multline*}}
yielding
{\allowdisplaybreaks\begin{eqnarray*}\allowdisplaybreaks
\Phi (f_{12}d_{21}^{*}+d_{21}f_{12}^{*})=\Phi (b_{12}d_{21}^{*}+d_{21}b_{12}^{*}).
\end{eqnarray*}}
This implies that $f_{12}d_{21}^{*}+d_{21}f_{12}^{*}=b_{12}d_{21}^{*}+d_{21}b_{12}^{*}$ which results in $f_{12}=b_{12}.$ Next, for an arbitrary element $d_{12}\in  \mathcal{A}_{12}$ we have
{\allowdisplaybreaks\begin{eqnarray*}\allowdisplaybreaks
&\Phi (d_{12}(f_{12}+f_{21})+(f_{12}+f_{21})d_{12}^{*})=\Phi (d_{12}b_{12}+b_{12}d_{12}^{*})+\Phi (d_{12}c_{21}+c_{21}d_{12}^{*})
\end{eqnarray*}}
which leads to
{\allowdisplaybreaks\begin{multline*}\allowdisplaybreaks
\Phi ((d_{12}(f_{12}+f_{21})+(f_{12}+f_{21})d_{12}^{*})p_{2}+p_{2}(d_{12}(f_{12}+f_{21})+(f_{12}+f_{21})d_{12}^{*})^{*})\\=\Phi ((d_{12}b_{12}+b_{12}d_{12}^{*})p_{2}+p_{2}(d_{12}b_{12}+b_{12}d_{12}^{*})^{*})\\+\Phi ((d_{12}c_{21}+c_{21}d_{12}^{*})p_{2}+p_{2}(d_{12}c_{21}+c_{21}d_{12}^{*})^{*})
\end{multline*}}
resulting in 
{\allowdisplaybreaks\begin{eqnarray*}\allowdisplaybreaks
\Phi (f_{21}d_{12}^{*}+d_{12}f_{21}^{*})=\Phi (c_{21}d_{12}^{*}+d_{12}c_{21}^{*}).
\end{eqnarray*}}
It follows that $f_{21}d_{12}^{*}+d_{12}f_{21}^{*}=c_{21}d_{12}^{*}+d_{12}c_{21}^{*}$ which show that $f_{21}=c_{21}.$

\noindent Second case: the product $ab-ba^{*}.$ Choose $f=f_{11}+f_{12}+f_{21}+f_{22}\in \mathcal{A}$ such that $\Phi (f)=\Phi (b_{12})+\Phi (c_{21}).$ Then 
{\allowdisplaybreaks\begin{eqnarray*}\allowdisplaybreaks
\Phi (fp_{2}-p_{2}f^{*})=\Phi (b_{12}p_{2}-p_{2}b_{12}^{*})+\Phi (c_{21}p_{2}-p_{2}c_{21}^{*})=\Phi (b_{12}-b_{12}^{*})
\end{eqnarray*}}
which implies that $fp_{2}-p_{2}f^{*}=b_{12}-b_{12}^{*}$ from which we get $f_{12}=b_{12}.$ Similarly we prove that $f_{21}=c_{21}.$
Next, for any element $d_{21}\in \mathcal{A}_{21}$ we have 
{\allowdisplaybreaks\begin{multline*}\allowdisplaybreaks
\Phi (d_{21}(f_{11}+b_{12}+c_{21}+f_{22})-(f_{11}+b_{12}+c_{21}+f_{22})d_{21}^{*})\\
=\Phi (d_{21}b_{12}-b_{12}d_{21}^{*})+\Phi (d_{21}c_{21}-c_{21}d_{21}^{*})
\end{multline*}}
from which we have
{\allowdisplaybreaks\begin{multline*}\allowdisplaybreaks
\Phi ((d_{21}f_{11}+d_{21}b_{12}+d_{21}c_{21}-f_{11}d_{21}^{*}-b_{12}d_{21}^{*}-c_{21}d_{21}^{*})p_{1}\\-p_{1}(d_{21}f_{11}+d_{21}b_{12}+d_{21}c_{21}-f_{11}d_{21}^{*}-b_{12}d_{21}^{*}-c_{21}d_{21}^{*})^{*})\\=\Phi ((d_{21}b_{12}-b_{12}d_{21}^{*})p_{1}-p_{1}(d_{21}b_{12}-b_{12}d_{21}^{*})^{*})\\+\Phi ((d_{21}c_{21}-c_{21}d_{21}^{*})p_{1}-p_{1}(d_{21}c_{21}-c_{21}d_{21}^{*})^{*})
\end{multline*}}
which leads to
{\allowdisplaybreaks\begin{eqnarray*}\allowdisplaybreaks
\Phi (d_{21}f_{11}-b_{12}d_{21}^{*}-f_{11}^{*}d_{21}^{*}+d_{21}b_{12}^{*})=\Phi (-b_{12}d_{21}^{*}+d_{21}b_{12}^{*}).
\end{eqnarray*}}
This results that $d_{21}f_{11}-b_{12}d_{21}^{*}-f_{11}^{*}d_{21}^{*}+d_{21}b_{12}^{*}=-b_{12}d_{21}^{*}+d_{21}b_{12}^{*}$ which yields that $d_{21}f_{11}=0$ and $f_{11}^{*}d_{21}^{*}=0.$ As consequence, we obtain $f_{11}=0.$ Next, for any element $d_{12}\in  \mathcal{A}_{12}$ we have
{\allowdisplaybreaks\begin{multline*}\allowdisplaybreaks
\Phi (d_{12}(b_{12}+c_{21}+f_{22})-(b_{12}+c_{21}+f_{22})d_{12}^{*})=\Phi (d_{12}b_{12}-b_{12}d_{12}^{*})\\+\Phi (d_{12}c_{21}-c_{21}d_{12}^{*})
\end{multline*}}
which leads to
{\allowdisplaybreaks\begin{multline*}\allowdisplaybreaks
\Phi ((d_{12}b_{12}+d_{12}c_{21}+d_{12}f_{22}-b_{12}d_{12}^{*}-c_{21}d_{12}^{*}-f_{22}d_{12}^{*})p_{2}\\
-p_{2}(d_{12}b_{12}+d_{12}c_{21}+d_{12}f_{22}-b_{12}d_{12}^{*}-c_{21}d_{12}^{*}-f_{22}d_{12}^{*})^{*})\\=\Phi ((d_{12}b_{12}-b_{12}d_{12}^{*})p_{2}-p_{2}(d_{12}b_{12}-b_{12}d_{12}^{*})^{*})\\+\Phi ((d_{12}c_{21}-c_{21}d_{12}^{*})p_{2}-p_{2}(d_{12}c_{21}-c_{21}d_{12}^{*})^{*})
\end{multline*}}
from which follows that
{\allowdisplaybreaks\begin{eqnarray*}\allowdisplaybreaks
\Phi (d_{12}f_{22}-c_{21}d_{12}^{*}-f_{22}^{*}d_{12}^{*}+d_{12}c_{21}^{*})=\Phi (-c_{21}d_{12}^{*}+d_{12}c_{21}^{*}).
\end{eqnarray*}}
This shows that $d_{12}f_{22}-c_{21}d_{12}^{*}-f_{22}^{*}d_{12}^{*}+d_{12}c_{21}^{*}=-c_{21}d_{12}^{*}+d_{12}c_{21}^{*}$ which implies that $d_{12}f_{22}=0$ and $f_{22}^{*}d_{12}^{*}=0.$ As consequence, we obtain $f_{22}=0.$
\end{proof}

\begin{claim}\label{claim23} For arbitrary elements $a_{11}\in \mathcal{A}_{11},$ $b_{12}\in \mathcal{A}_{12},$ $c_{21}\in \mathcal{A}_{21}$ and $d_{22}\in \mathcal{A}_{22}$ hold: (i) $\Phi (a_{11}+b_{12}+c_{21})=\Phi (a_{11})+\Phi (b_{12})+\Phi (c_{21})$ and (ii) $\Phi (b_{12}+c_{21}+d_{22})=\Phi (b_{12})+\Phi (c_{21})+\Phi (d_{22}).$
\end{claim}
\begin{proof} First case: the product $ab+ba^{*}.$ Choose an element $f=f_{11}+f_{12}+f_{21}+f_{22}\in \mathcal{A}$ such that $\Phi (f)=\Phi (a_{11}) + \Phi (b_{12})+\Phi (c_{21})$. By Claim \ref{lem22} we have
{\allowdisplaybreaks\begin{multline*}\allowdisplaybreaks
\Phi (p_{2}f+fp_{2}^{*})=\Phi (p_{2}a_{11}+a_{11}p_{2}^{*})+\Phi (p_{2}b_{12}+b_{12}p_{2}^{*})+\Phi (p_{2}c_{21}+c_{21}p_{2}^{*})\\
=\Phi (b_{12})+\Phi (c_{21})=\Phi (b_{12}+c_{21}).
\end{multline*}}
This implies that $p_{2}f+fp_{2}^{*}=b_{12}+c_{21}$ resulting in $f_{12}=b_{12},$ $f_{21}=c_{21}$ and $f_{22}=0.$ It follows that
\begin{eqnarray*}
\Phi (f_{11}+b_{12}+c_{21})=\Phi (a_{11})+\Phi (b_{12})+\Phi (c_{21})=\Phi (a_{11}+b_{12})+\Phi (c_{21}),
\end{eqnarray*}
by Claim \ref{claim21}(i). Hence, for an arbitrary element $d_{21}\in \mathcal{A}_{21}$ we have 
\begin{multline*}
\Phi (d_{21}(f_{11}+b_{12}+c_{21})+(f_{11}+b_{12}+c_{21})d_{21}^{*})\\
=\Phi (d_{21}(a_{11}+b_{12})+(a_{11}+b_{12})d_{21}^{*})+\Phi (d_{21}c_{21}+c_{21}d_{21}^{*})
\end{multline*}
which leads to
\begin{multline*}
\Phi (d_{21}f_{11}+d_{21}b_{12}+d_{21}c_{21}+f_{11}d_{21}^{*}+b_{12}d_{21}^{*}+c_{21}d_{21}^{*})\\
=\Phi (d_{21}a_{11}+d_{21}b_{12}+a_{11}d_{21}^{*}+b_{12}d_{21}^{*})+\Phi (d_{21}c_{21}+c_{21}d_{21}^{*}).
\end{multline*}
It follows that
\begin{multline*}
\Phi ((d_{21}f_{11}+d_{21}b_{12}+d_{21}c_{21}+f_{11}d_{21}^{*}+b_{12}d_{21}^{*}+c_{21}d_{21}^{*})p_{1}\\
+p_{1}(d_{21}f_{11}+d_{21}b_{12}+d_{21}c_{21}+f_{11}d_{21}^{*}+b_{12}d_{21}^{*}+c_{21}d_{21}^{*})^{*})\\
=\Phi ((d_{21}a_{11}+d_{21}b_{12}+a_{11}d_{21}^{*}+b_{12}d_{21}^{*})p_{1}+p_{1}(d_{21}a_{11}+d_{21}b_{12}+a_{11}d_{21}^{*}+b_{12}d_{21}^{*})^{*})\\
+\Phi ((d_{21}c_{21}+c_{21}d_{21}^{*})p_{1}+p_{1}(d_{21}c_{21}+c_{21}d_{21}^{*})^{*})
\end{multline*}
which implies that
\begin{eqnarray*}
&\Phi (d_{21}f_{11}+b_{12}d_{21}^{*}+f_{11}^{*}d_{21}^{*}+d_{21}b_{12}^{*})=\Phi (d_{21}a_{11}+a_{11}^{*}d_{21}^{*}+b_{12}d_{21}^{*}+d_{21}b_{12}^{*}).
\end{eqnarray*}
This results in $d_{21}f_{11}+b_{12}d_{21}^{*}+f_{11}^{*}d_{21}^{*}+d_{21}b_{12}^{*}=d_{21}a_{11}+a_{11}^{*}d_{21}^{*}+b_{12}d_{21}^{*}+d_{21}b_{12}^{*}$ which yields that $f_{11}=a_{11}.$

Similarly, we prove the case (ii).

\noindent Second case: the product $ab-ba^{*}.$ Choose $f=f_{11}+f_{12}+f_{21}+f_{22}\in \mathcal{A}$ such that $\Phi (f)=\Phi (a_{11}) + \Phi (b_{12})+\Phi (c_{21})$. By Claim \ref{lem22} we have
{\allowdisplaybreaks\begin{multline*}\allowdisplaybreaks
\Phi (p_{2}f-fp_{2}^{*})=\Phi (p_{2}a_{11}-a_{11}p_{2}^{*})+\Phi (p_{2}b_{12}-b_{12}p_{2}^{*})+\Phi (p_{2}c_{21}-c_{21}p_{2}^{*})\\
=\Phi (-b_{12})+\Phi (c_{21})=\Phi (-b_{12}+c_{21}).
\end{multline*}}
This implies that $p_{2}f-fp_{2}^{*}=-b_{12}+c_{21}$ from which we obtain $f_{12}=b_{12}$ and $f_{21}=c_{21}.$ We have thus shown that 
\begin{eqnarray*}
\Phi (f_{11}+b_{12}+c_{21}+f_{22})=\Phi (a_{11})+\Phi (b_{12})+\Phi (c_{21})=\Phi (a_{11}+b_{12})+\Phi (c_{21}),
\end{eqnarray*}
by Claim \ref{claim21}(i). Hence, for any element $d_{21}\in \mathcal{A}_{21}$ we have 
\begin{multline*}
\Phi (d_{21}(f_{11}+b_{12}+c_{21}+f_{22})-(f_{11}+b_{12}+c_{21}+f_{22})d_{21}^{*})\\
=\Phi (d_{21}(a_{11}+b_{12})-(a_{11}+b_{12})d_{21}^{*})+\Phi (d_{21}c_{21}-c_{21}d_{21}^{*})
\end{multline*}
from which follows that
\begin{multline*}
\Phi (d_{21}f_{11}+d_{21}b_{12}+d_{21}c_{21}-f_{11}d_{21}^{*}-b_{12}d_{21}^{*}-c_{21}d_{21}^{*})\\
=\Phi (d_{21}a_{11}+d_{21}b_{12}-a_{11}d_{21}^{*}-b_{12}d_{21}^{*})+\Phi (d_{21}c_{21}-c_{21}d_{21}^{*}).
\end{multline*}
As consequence, we obtain
\begin{multline*}
\Phi ((d_{21}f_{11}+d_{21}b_{12}+d_{21}c_{21}-f_{11}d_{21}^{*}-b_{12}d_{21}^{*}-c_{21}d_{21}^{*})p_{1}\\
-p_{1}(d_{21}f_{11}+d_{21}b_{12}+d_{21}c_{21}-f_{11}d_{21}^{*}-b_{12}d_{21}^{*}-c_{21}d_{21}^{*})^{*})\\
=\Phi ((d_{21}a_{11}+d_{21}b_{12}-a_{11}d_{21}^{*}-b_{12}d_{21}^{*})p_{1}-p_{1}(d_{21}a_{11}+d_{21}b_{12}-a_{11}d_{21}^{*}-b_{12}d_{21}^{*})^{*})\\
+\Phi ((d_{21}c_{21}-c_{21}d_{21}^{*})p_{1}-p_{1}(d_{21}c_{21}-c_{21}d_{21}^{*})^{*})
\end{multline*}
which implies that
\begin{eqnarray*}
&\Phi (d_{21}f_{11}-b_{12}d_{21}^{*}-f_{11}^{*}d_{21}^{*}+d_{21}b_{12}^{*})=\Phi (d_{21}a_{11}-b_{12}d_{21}^{*}-a_{11}^{*}d_{21}^{*}+d_{21}b_{12}^{*}).
\end{eqnarray*}
It follows that $d_{21}f_{11}-b_{12}d_{21}^{*}-f_{11}^{*}d_{21}^{*}+d_{21}b_{12}^{*}=d_{21}a_{11}-b_{12}d_{21}^{*}-a_{11}^{*}d_{21}^{*}+d_{21}b_{12}^{*}$ which yields that $d_{21}f_{11}=d_{21}a_{11}$ and $f_{11}^{*}d_{21}^{*}=a_{11}^{*}d_{21}^{*}.$ This result gives us $f_{11}=a_{11}$ which allows us to write $\Phi (a_{11}+b_{12}+c_{21}+f_{22})=\Phi (a_{11})+\Phi (b_{12})+\Phi (c_{21})=\Phi (a_{11}+b_{12})+\Phi (c_{21}),$ by Claim \ref{claim21}(i). Hence, for any element $d_{12}\in \mathcal{A}_{21}$ we have 
\begin{multline*}
\Phi (d_{12}(a_{11}+b_{12}+c_{21}+f_{22})-(a_{11}+b_{12}+c_{21}+f_{22})d_{12}^{*})\\
=\Phi (d_{12}(a_{11}+b_{12})-(a_{11}+b_{12})d_{12}^{*})+\Phi (d_{12}c_{21}-c_{21}d_{12}^{*})
\end{multline*}
which leads to
\begin{multline*}
\Phi (d_{12}b_{12}+d_{12}c_{21}+d_{12}f_{22}-b_{12}d_{12}^{*}-c_{21}d_{12}^{*}-f_{22}d_{12}^{*})\\
=\Phi (d_{12}b_{12}-b_{12}d_{12}^{*})+\Phi (d_{12}c_{21}-c_{21}d_{12}^{*}).
\end{multline*}
It follows that
\begin{multline*}
\Phi ((d_{12}b_{12}+d_{12}c_{21}+d_{12}f_{22}-b_{12}d_{12}^{*}-c_{21}d_{12}^{*}-f_{22}d_{12}^{*})p_{2}\\
-p_{2}(d_{12}b_{12}+d_{12}c_{21}+d_{12}f_{22}-b_{12}d_{12}^{*}-c_{21}d_{12}^{*}-f_{22}d_{12}^{*})^{*})\\
=\Phi ((d_{12}b_{12}-b_{12}d_{12}^{*})p_{2}-p_{2}(d_{12}b_{12}-b_{12}d_{12}^{*})^{*})\\
+\Phi ((d_{12}c_{21}-c_{21}d_{12}^{*})p_{2}-p_{2}(d_{12}c_{21}-c_{21}d_{12}^{*})^{*})
\end{multline*}
from which we get
\begin{eqnarray*}
&\Phi (d_{12}f_{22}-c_{21}d_{12}^{*}
-f_{22}^{*}d_{12}^{*}+d_{12}c_{21}^{*})=\Phi (-c_{21}^{*}d_{12}^{*}+d_{12}c_{21}^{*}).
\end{eqnarray*}
This shows that $d_{12}f_{22}-c_{21}d_{12}^{*}-f_{22}^{*}d_{12}^{*}+d_{12}c_{21}^{*}=-c_{21}^{*}d_{12}^{*}+d_{12}c_{21}^{*}$ which yields that $d_{12}f_{22}=0$ and $f_{22}^{*}d_{12}^{*}=0.$ As consequence, we obtain $f_{22}=0.$

Similarly, we prove the case (ii).
\end{proof}

\begin{claim}\label{claim24} For arbitrary elements $a_{11}\in \mathcal{A}_{11},$ $b_{12}\in \mathcal{A}_{12},$ $c_{21}\in \mathcal{A}_{21}$ and $d_{22}\in \mathcal{A}_{22}$ holds $\Phi (a_{11}+b_{12}+c_{21}+d_{22})=\Phi (a_{11})+\Phi (b_{12})+\Phi (c_{21})+\Phi (d_{22}).$
\end{claim}
\begin{proof} First case: the product $ab+ba^{*}.$ Choose $f=f_{11}+f_{12}+f_{21}+f_{22}\in \mathcal{A}$ such that $\Phi (f)=\Phi (a_{11}) + \Phi (b_{12})+\Phi (c_{21})+\Phi (d_{22})$. By Claim \ref{lem23}(i) we get
{\allowdisplaybreaks\begin{multline*}\allowdisplaybreaks
\Phi (p_{1}f+fp_{1}^{*})=\Phi (p_{1}a_{11}+a_{11}p_{1}^{*})+\Phi (p_{1}b_{12}+b_{12}p_{1}^{*})+\Phi (p_{1}c_{21}+c_{21}p_{1}^{*})\\
+\Phi (p_{1}d_{22}+d_{22}p_{1}^{*})
=\Phi (2a_{11})+\Phi (b_{12})+\Phi (c_{21})=\Phi (2a_{11}+b_{12}+c_{21})
\end{multline*}}
which leads to $p_{1}f+fp_{1}^{*}=2a_{11}+b_{12}+c_{21}.$ As consequence we obtain $f_{11}=a_{11},$ $f_{12}=b_{12}$ and $f_{21}=c_{21}.$ Next, by Claim \ref{lem23}(ii) we have
{\allowdisplaybreaks\begin{multline*}\allowdisplaybreaks
\Phi (p_{2}f+fp_{2}^{*})=\Phi (p_{2}a_{11}+a_{11}p_{2}^{*})+\Phi (p_{2}b_{12}+b_{12}p_{2}^{*})+\Phi (p_{2}c_{21}+c_{21}p_{2}^{*})\\+\Phi (p_{2}d_{22}+d_{22}p_{2}^{*})
=\Phi (b_{12})+\Phi (c_{21})+\Phi (2d_{22})=\Phi (b_{12}+c_{21}+2d_{22}).
\end{multline*}}
This implies that $p_{2}f+fp_{2}^{*}=b_{12}+c_{21}+2d_{22}$ which results in $f_{22}=d_{22}.$

\noindent Second case: the product $ab-ba^{*}.$ Choose $f=f_{11}+f_{12}+f_{21}+f_{22}\in \mathcal{A}$ such that $\Phi (f)=\Phi (a_{11}) + \Phi (b_{12})+\Phi (c_{21})+\Phi (d_{22})$. By Claim \ref{lem23}(i) we get
{\allowdisplaybreaks\begin{multline*}\allowdisplaybreaks
\Phi ((ip_{1})f-f(ip_{1})^{*})=\Phi ((ip_{1})a_{11}-a_{11}(ip_{1})^{*})+\Phi ((ip_{1})b_{12}-b_{12}(ip_{1})^{*})\\+\Phi ((ip_{1})c_{21}-c_{21}(ip_{1})^{*})+\Phi ((ip_{1})d_{22}-d_{22}(ip_{1})^{*})\\
=\Phi (2ia_{11})+\Phi (ib_{12})+\Phi (ic_{21})=\Phi (2ia_{11}+ib_{12}+ic_{21})
\end{multline*}}
which leads to $(ip_{1})f-f(ip_{1})^{*}=2ia_{11}+ib_{12}+ic_{21}.$ As consequence we obtain $f_{11}=a_{11},$ $f_{12}=b_{12}$ and $f_{21}=c_{21}.$ Next, by Claim \ref{lem23}(ii) we have
{\allowdisplaybreaks\begin{multline*}\allowdisplaybreaks
\Phi ((ip_{2})f-f(ip_{2})^{*})=\Phi ((ip_{2})a_{11}-a_{11}(ip_{2})^{*})+\Phi ((ip_{2})b_{12}-b_{12}(ip_{2})^{*})\\
+\Phi ((ip_{2})c_{21}-c_{21}(ip_{2})^{*})+\Phi ((ip_{2})d_{22}-d_{22}(ip_{2})^{*})\\
=\Phi (ib_{12})+\Phi (ic_{21})+\Phi (2id_{22})=\Phi (ib_{12}+ic_{21}+2id_{22})
\end{multline*}}
from which we have that $p_{2}f-fp_{2}^{*}=ib_{12}+ic_{21}+2id_{22}.$ Consequently, we conclude that $f_{22}=d_{22}.$
\end{proof}

\begin{claim}\label{claim25} For arbitrary elements $a_{12},b_{12}\in \mathcal{A}_{12}$ and $c_{21},d_{21}\in \mathcal{A}_{21}$ hold: (i) $\Phi (a_{12}b_{12}+a_{12}^{*})=\Phi (a_{12}b_{12})+ \Phi (a_{12}^{*})$ and (ii) $\Phi (c_{21}d_{21}+c_{21}^{*})=\Phi (c_{21}d_{21})+ \Phi (c_{21}^{*})$ (resp., (i) $\Phi (a_{12}b_{12}-a_{12}^{*})=\Phi (a_{12}b_{12})+ \Phi (-a_{12}^{*})$ and (ii) $\Phi (c_{21}d_{21}-c_{21}^{*})=\Phi (c_{21}d_{21})+ \Phi (-c_{21}^{*})$).
\end{claim}
\begin{proof} First case: the product $ab+ba^{*}.$ First, we note that the following identity is valid
{\allowdisplaybreaks\begin{eqnarray*}\allowdisplaybreaks
a_{12}(p_{2}+b_{12})+(p_{2}+b_{12})a_{12}^{*}=a_{12}+a_{12}b_{12}+a_{12}^{*}+b_{12}a_{12}^{*}.
\end{eqnarray*}}
Hence, by Claim \ref{claim24} we have
{\allowdisplaybreaks\begin{eqnarray*}\allowdisplaybreaks
&&\Phi (a_{12})+\Phi (a_{12}b_{12}+a_{12}^{*})+\Phi (b_{12}a_{12}^{*})\\
&=&\Phi (a_{12}+a_{12}b_{12}+a_{12}^{*}+b_{12}a_{12}^{*})\\
&=&\Phi (a_{12}(p_{2}+b_{12})+(p_{2}+b_{12})a_{12}^{*})\\
&=&\Phi (a_{12})\Phi (p_{2}+b_{12})+\Phi (p_{2}+b_{12})\Phi (a_{12})^{*}\\
&=&\Phi (a_{12})(\Phi (p_{2})+\Phi (b_{12}))+(\Phi (p_{2})+\Phi (b_{12}))\Phi (a_{12})^{*}\\
&=&\Phi (a_{12})\Phi (p_{2})+\Phi (p_{2})\Phi (a_{12})^{*}+\Phi (a_{12})\Phi (b_{12})+\Phi (b_{12})\Phi (a_{12})^{*}\\
&=&\Phi (a_{12}p_{2}+p_{2}a_{12}^{*})+\Phi (a_{12}b_{12}+b_{12}a_{12}^{*})\\
&=&\Phi (a_{12})+\Phi (a_{12}^{*})+\Phi (a_{12}b_{12})+\Phi (b_{12}a_{12}^{*}).
\end{eqnarray*}}
This allows us to conclude that $\Phi (a_{12}b_{12}+a_{12}^{*})=\Phi (a_{12}b_{12})+ \Phi (a_{12}^{*}).$

Similarly, we prove the case (ii) using the identity
{\allowdisplaybreaks\begin{eqnarray*}\allowdisplaybreaks
c_{21}(p_{1}+d_{21})+(p_{1}+d_{21})c_{21}^{*}=c_{21}+c_{21}d_{21}+c_{21}^{*}+d_{21}c_{21}^{*}.
\end{eqnarray*}}

\noindent Second case: the product $ab-ba^{*}.$ First, we note that the following identity holds
{\allowdisplaybreaks\begin{eqnarray*}\allowdisplaybreaks
a_{12}(p_{2}+b_{12})-(p_{2}+b_{12})a_{12}^{*}=a_{12}+a_{12}b_{12}-a_{12}^{*}-b_{12}a_{12}^{*}.
\end{eqnarray*}}
Hence, by Claim \ref{claim24} we have
{\allowdisplaybreaks\begin{eqnarray*}\allowdisplaybreaks
&&\Phi (a_{12})+\Phi (a_{12}b_{12}-a_{12}^{*})+\Phi (-b_{12}a_{12}^{*})\\
&=&\Phi (a_{12}+a_{12}b_{12}-a_{12}^{*}-b_{12}a_{12}^{*})\\
&=&\Phi (a_{12}(p_{2}+b_{12})-(p_{2}+b_{12})a_{12}^{*})\\
&=&\Phi (a_{12})\Phi (p_{2}+b_{12})-\Phi (p_{2}+b_{12})\Phi (a_{12})^{*}\\
&=&\Phi (a_{12})(\Phi (p_{2})+\Phi (b_{12}))-(\Phi (p_{2})+\Phi (b_{12}))\Phi (a_{12})^{*}\\
&=&\Phi (a_{12})\Phi (p_{2})-\Phi (p_{2})\Phi (a_{12})^{*}+\Phi (a_{12})\Phi (b_{12})-\Phi (b_{12})\Phi (a_{12})^{*}\\
&=&\Phi (a_{12}p_{2}-p_{2}a_{12}^{*})+\Phi (a_{12}b_{12}-b_{12}a_{12}^{*})\\
&=&\Phi (a_{12})+\Phi (-a_{12}^{*})+\Phi (a_{12}b_{12})+\Phi (-b_{12}a_{12}^{*}).
\end{eqnarray*}}
This allows us to conclude that $\Phi (a_{12}b_{12}-a_{12}^{*})=\Phi (a_{12}b_{12})+ \Phi (-a_{12}^{*}).$

Similarly, we prove the case (ii) using the identity
{\allowdisplaybreaks\begin{eqnarray*}\allowdisplaybreaks
c_{21}(p_{1}+d_{21})-(p_{1}+d_{21})c_{21}^{*}=c_{21}+c_{21}d_{21}-c_{21}^{*}-d_{21}c_{21}^{*}.
\end{eqnarray*}}
\end{proof}

\begin{claim}\label{claim26} For arbitrary elements $a_{12},b_{12}\in \mathcal{A}_{12}$ and $c_{21},d_{21}\in \mathcal{A}_{21}$ hold: (i) $\Phi (a_{12}+b_{12})=\Phi (a_{12})+ \Phi (b_{12})$ and (ii) $\Phi (c_{21}+d_{21})=\Phi (c_{21})+ \Phi (d_{21}).$
\end{claim}
\begin{proof} First case: the product $ab+ba^{*}.$ First, we note that the following identity is valid
{\allowdisplaybreaks\begin{multline*}\allowdisplaybreaks
(p_{1}+a_{12})(p_{2}+b_{12})+(p_{2}+b_{12})(p_{1}+a_{12})^{*}=a_{12}+b_{12}+a_{12}b_{12}+a_{12}^{*}+b_{12}a_{12}^{*}.
\end{multline*}}
Hence, by Claims \ref{claim24} and \ref{claim25} we have
{\allowdisplaybreaks\begin{eqnarray*}\allowdisplaybreaks
&&\Phi (a_{12}+b_{12})+\Phi (a_{12}b_{12})+\Phi (a_{12}^{*})+\Phi (b_{12}a_{12}^{*})\\
&=&\Phi (a_{12}+b_{12})+\Phi (a_{12}b_{12}+a_{12}^{*})+\Phi (b_{12}a_{12}^{*})\\
&=&\Phi (a_{12}+b_{12}+a_{12}b_{12}+a_{12}^{*}+b_{12}a_{12}^{*})\\
&=&\Phi ((p_{1}+a_{12})(p_{2}+b_{12})+(p_{2}+b_{12})(p_{1}+a_{12})^{*})\\
&=&\Phi (p_{1}+a_{12})\Phi (p_{2}+b_{12})+\Phi (p_{2}+b_{12})\Phi (p_{1}+a_{12})^{*}\\
&=&(\Phi (p_{1})+\Phi (a_{12}))(\Phi (p_{2})+\Phi (b_{12}))\\
&&+(\Phi (p_{2})+\Phi (b_{12}))(\Phi (p_{1})^{*}+\Phi (a_{12})^{*})\\
&=&\Phi (p_{1})\Phi (p_{2})+\Phi (p_{2})\Phi (p_{1})^{*}+\Phi (p_{1})\Phi (b_{12})+\Phi (b_{12})\Phi (p_{1})^{*}\\
&&+\Phi (a_{12})\Phi (p_{2})+\Phi (p_{2})\Phi (a_{12})^{*}+\Phi (a_{12})\Phi (b_{12})+\Phi (b_{12})\Phi (a_{12})^{*}\\
&=&\Phi (p_{1}p_{2}+p_{2}p_{1}^{*})+\Phi (p_{1}b_{12}+b_{12}p_{1}^{*})+\Phi (a_{12}p_{2}+p_{2}a_{12}^{*})\\
&&+\Phi (a_{12}b_{12}+b_{12}a_{12}^{*})\\
&=&\Phi (b_{12})+\Phi (a_{12})+\Phi (a_{12}^{*})+\Phi (a_{12}b_{12})+\Phi (b_{12}a_{12}^{*}).
\end{eqnarray*}}
This allows us to conclude that $\Phi (a_{12}+b_{12})=\Phi (a_{12})+ \Phi (b_{12}).$

Similarly, we prove the case (ii) using the identity
{\allowdisplaybreaks\begin{multline*}\allowdisplaybreaks
(p_{2}+c_{21})(p_{1}+d_{21})+(p_{1}+d_{21})(p_{2}+c_{21})^{*}=c_{21}+d_{21}+c_{21}d_{21}+c_{21}^{*}+d_{21}c_{21}^{*}.
\end{multline*}}

\noindent Second case: the product $ab-ba^{*}.$ First, we note that the following identity holds
{\allowdisplaybreaks\begin{multline*}\allowdisplaybreaks
(p_{1}+a_{12})(p_{2}+b_{12})-(p_{2}+b_{12})(p_{1}+a_{12})^{*}=a_{12}+b_{12}+a_{12}b_{12}-a_{12}^{*}-b_{12}a_{12}^{*}.
\end{multline*}}
Hence, by Claims \ref{claim24} and \ref{claim25} we have
{\allowdisplaybreaks\begin{eqnarray*}\allowdisplaybreaks
&&\Phi (a_{12}+b_{12})+\Phi (a_{12}b_{12})+\Phi (-a_{12}^{*})+\Phi (-b_{12}a_{12}^{*})\\
&=&\Phi (a_{12}+b_{12})+\Phi (a_{12}b_{12}-a_{12}^{*})+\Phi (-b_{12}a_{12}^{*})\\
&=&\Phi (a_{12}+b_{12}+a_{12}b_{12}-a_{12}^{*}-b_{12}a_{12}^{*})\\
&=&\Phi ((p_{1}+a_{12})(p_{2}+b_{12})-(p_{2}+b_{12})(p_{1}+a_{12})^{*})\\
&=&\Phi (p_{1}+a_{12})\Phi (p_{2}+b_{12})-\Phi (p_{2}+b_{12})\Phi (p_{1}+a_{12})^{*}\\
&=&(\Phi (p_{1})+\Phi (a_{12}))(\Phi (p_{2})+\Phi (b_{12}))\\
&&-(\Phi (p_{2})+\Phi (b_{12}))(\Phi (p_{1})^{*}+\Phi (a_{12})^{*})\\
&=&\Phi (p_{1})\Phi (p_{2})-\Phi (p_{2})\Phi (p_{1})^{*}+\Phi (p_{1})\Phi (b_{12})-\Phi (b_{12})\Phi (p_{1})^{*}\\
&&+\Phi (a_{12})\Phi (p_{2})-\Phi (p_{2})\Phi (a_{12})^{*}+\Phi (a_{12})\Phi (b_{12})-\Phi (b_{12})\Phi (a_{12})^{*}\\
&=&\Phi (p_{1}p_{2}-p_{2}p_{1}^{*})+\Phi (p_{1}b_{12}-b_{12}p_{1}^{*})+\Phi (a_{12}p_{2}-p_{2}a_{12}^{*})\\
&&+\Phi (a_{12}b_{12}-b_{12}a_{12}^{*})\\
&=&\Phi (b_{12})+\Phi (a_{12})+\Phi (-a_{12}^{*})+\Phi (a_{12}b_{12})+\Phi (-b_{12}a_{12}^{*}).
\end{eqnarray*}}
This allows us to conclude that $\Phi (a_{12}+b_{12})=\Phi (a_{12})+ \Phi (b_{12}).$

Similarly, we prove the case (ii) using the identity
{\allowdisplaybreaks\begin{multline*}\allowdisplaybreaks
(p_{2}+c_{21})(p_{1}+d_{21})-(p_{1}+d_{21})(p_{2}+c_{21})^{*}=c_{21}+d_{21}+c_{21}d_{21}-c_{21}^{*}-d_{21}c_{21}^{*}.
\end{multline*}}
\end{proof}

\begin{claim}\label{claim27} For arbitrary elements $a_{11},b_{11}\in \mathcal{A}_{11}$ and $c_{22},d_{22}\in \mathcal{A}_{22}$ hold: (i) $\Phi (a_{11}+b_{11})=\Phi (a_{11})+\Phi (b_{11})$ and (ii) $\Phi (c_{22}+d_{22})=\Phi (c_{22})+\Phi (d_{22}).$
\end{claim}
\begin{proof} First case: the product $ab+ba^{*}.$ Choose $f=f_{11}+f_{12}+f_{21}+f_{22}\in \mathcal{A}$ such that $\Phi (f)=\Phi (a_{11}) + \Phi (b_{11})$. By Claim \ref{claim26}(i) we have
{\allowdisplaybreaks\begin{eqnarray*}\allowdisplaybreaks
\Phi (p_{2}f+fp_{2}^{*})=\Phi (p_{2}a_{11}+a_{11}p_{2}^{*})+\Phi (p_{2}b_{11}+b_{11}p_{2}^{*})=0.
\end{eqnarray*}}
This results that $p_{2}f+fp_{2}^{*}=0$ which implies that $f_{12}=0,$ $f_{21}=0$ and $f_{22}=0.$ It follows that $\Phi (f_{11})=\Phi (a_{11})+\Phi (b_{11}).$ Hence, for an arbitrary element $d_{21}\in \mathcal{A}_{21}$ we have 
{\allowdisplaybreaks\begin{eqnarray*}\allowdisplaybreaks
&&\Phi (d_{21}f_{11}+f_{11}d_{21}^{*})\\
&=&\Phi (d_{21}a_{11}+a_{11}d_{21}^{*})+\Phi (d_{21}b_{11}+b_{11}d_{21}^{*})\\
&=&\Phi (d_{21}(a_{11}+b_{11})+(a_{11}+b_{11})d_{21}^{*}),
\end{eqnarray*}}
by Claims \ref{claim24} and \ref{claim26}. This implies that $d_{21}f_{11}+f_{11}d_{21}^{*}=d_{21}(a_{11}+b_{11})+(a_{11}+b_{11})d_{21}^{*}$ which results in $f_{11}=a_{11}+b_{11}.$

Similarly, we prove the case (ii).

\noindent Second case: the product $ab-ba^{*}.$ Choose $f=f_{11}+f_{12}+f_{21}+f_{22}\in \mathcal{A}$ such that $\Phi (f)=\Phi (a_{11}) + \Phi (b_{11})$. Then {\allowdisplaybreaks\begin{eqnarray*}\allowdisplaybreaks
\Phi (p_{1}f-fp_{1}^{*})=\Phi (p_{1}a_{11}-a_{11}p_{1}^{*})+\Phi (p_{1}b_{11}-b_{11}p_{1}^{*})=0
\end{eqnarray*}}
from which we get $p_{1}f-fp_{1}^{*}=0$ which leads to $f_{12}=0$ and $f_{21}=0.$ This show that $\Phi (f_{11}+f_{22})=\Phi (a_{11})+\Phi (b_{11}).$ Hence, for any element $d_{21}\in \mathcal{A}_{21}$ we have 
{\allowdisplaybreaks\begin{eqnarray*}\allowdisplaybreaks
&&\Phi (d_{21}(f_{11}+f_{22})-(f_{11}+f_{22})d_{21}^{*})\\
&=&\Phi (d_{21}a_{11}-a_{11}d_{21}^{*})+\Phi (d_{21}b_{11}-b_{11}d_{21}^{*})\\
&=&\Phi (d_{21}(a_{11}+b_{11})-(a_{11}+b_{11})d_{21}^{*}),
\end{eqnarray*}}
by Claims \ref{claim24} and \ref{claim26}. This implies that $d_{21}(f_{11}+f_{22})-(f_{11}+f_{22})d_{21}^{*}=d_{21}(a_{11}+b_{11})-(a_{11}+b_{11})d_{21}^{*}$ which results that $d_{21}f_{11}=d_{21}(a_{11}+b_{11})$ and $f_{11}d_{21}^{*}=(a_{11}+b_{11})d_{21}^{*}.$ As consequence, we obtain $f_{11}=a_{11}+b_{11}.$ It follows that $\Phi (a_{11}+f_{22})=\Phi (a_{11})+\Phi (b_{11}).$ Hence, for any element $d_{12}\in \mathcal{A}_{21}$ we have 
{\allowdisplaybreaks\begin{eqnarray*}\allowdisplaybreaks
&&\Phi (d_{12}(a_{11}+f_{22})-(a_{11}+f_{22})d_{12}^{*})\\
&=&\Phi (d_{12}a_{11}-a_{11}d_{12}^{*})+\Phi (d_{12}b_{11}-b_{11}d_{12}^{*})\\
&=&0.
\end{eqnarray*}}
This implies that $d_{12}(a_{11}+f_{22})-(a_{11}+f_{22})d_{12}^{*}=0$ which results that $d_{12}f_{22}=0$ and $f_{22}d_{12}^{*}=0.$ As consequence, we obtain $f_{22}=0.$

Similarly, we prove the case (ii).
\end{proof}

\begin{claim}\label{claim28} $\Phi $ is an additive mapping.
\end{claim}
\begin{proof} The result is an immediate consequence of Claims \ref{claim24}, \ref{claim26} and \ref{claim27}.
\end{proof}

\begin{lemma}\label{lem24} Let $\mathcal{B}$ be an alternative $W^{*}$-factor with identity $1_{\mathcal{B}}$ and an element $u\in \mathcal{B}.$ If $\mathcal{B}$ satisfies the condition $uv+vu^{*}=0,$ for all element $v\in \mathcal{B},$ then $u\in i\mathds{R}1_{\mathcal{B}}$ ($i$ is the imaginary number unit) (resp., If $\mathcal{B}$ satisfies the condition $uv-vu^{*}=0,$ for all element $v\in \mathcal{B},$ then $u\in \mathds{R}1_{\mathcal{B}}.$).
\end{lemma}
\begin{proof} First case: the product $ab+ba^{*}.$ Taking $v=1_{\mathcal{B}}$ we get $u+u^{*}=0.$ Replacing this last result in given condition, we obtain $uv-vu=0,$ for all element $v\in \mathcal{B}.$ This implies that $u$ belongs to
the center of $\mathcal{B}$ which results that $u\in i\mathds{R}1_{\mathcal{B}}.$

\noindent Second case: the product $ab-ba^{*}.$ Replacing $v$ with $1_{\mathcal{B}}$ we get $u-u^{*}=0.$ It follows that $uv-vu=0,$ for all element $v\in \mathcal{B},$ which implies that $u$ belongs to the center of $\mathcal{B}.$ This results that $u\in \mathds{R}1_{\mathcal{B}}.$
\end{proof}

In the remainder of this paper, we prove that $\Phi $ is a $\ast $-ring isomorphism. From now on, we assume that all lemmas satisfy the conditions of the Main Theorem.

\noindent First case: the product $ab+ba^{*}.$ 

\begin{lemma}\label{lem25} $\Phi (i\mathds{R}1_{\mathcal{A}})=i\mathds{R}1_{\mathcal{B}}$ and $\Phi (\mathds{C}1_{\mathcal{A}})=\mathds{C}1_{\mathcal{B}}.$
\end{lemma}
\begin{proof} For arbitrary elements $\lambda \in \mathds{R}$ and $a\in \mathcal{A}$ we have
{\allowdisplaybreaks\begin{eqnarray*}\allowdisplaybreaks
0=\Phi ((i\lambda 1_{\mathcal{A}})a+a(i\lambda 1_{\mathcal{A}})^{*})=\Phi (i\lambda 1_{\mathcal{A}})\Phi (a)+\Phi (a)\Phi (i\lambda 1_{\mathcal{A}})^{*}.
\end{eqnarray*}}
By Lemma \ref{lem24} we obtain $\Phi (i\lambda 1_{\mathcal{A}})\in i\mathds{R}1_{\mathcal{B}}.$ Since $\lambda $ is an arbitrary real number, then we conclude that $\Phi (i\mathds{R}1_{\mathcal{A}})\subseteq i\mathds{R}1_{\mathcal{B}}.$ Note that $\Phi ^{-1}$ has the same properties that $\Phi .$ Applying a similar argument to the above we also conclude that $i\mathds{R}1_{\mathcal{B}}\subseteq \Phi (i\mathds{R}1_{\mathcal{A}}).$ As consequence, we obtain $\Phi (i\mathds{R}1_{\mathcal{A}})=i\mathds{R}1_{\mathcal{B}}.$

Let $a\in \mathcal{A}$ an element satisfying $a^{*}=-a$ (called an {\it anti-self-adjoint element}). Then $0=\Phi (a(i1_{\mathcal{A}})+(i1_{\mathcal{A}})a^{*})=\Phi (a)\Phi (i1_{\mathcal{A}})+\Phi (i1_{\mathcal{A}})\Phi (a)^{*}$ which implies that $\Phi (a)^{*}=-\Phi (a).$ Since $\Phi ^{-1}$ has the same properties of $\Phi ,$ then applying a similar reasoning we also prove that for any element $a\in \mathcal{A},$ satisfying $\Phi (a)^{*}=-\Phi (a),$ then $a^{*}=-a.$ Consequently, $a=-a^{*}$ if and only if $\Phi (a)=-\Phi (a)^{*},$ for any such element $a\in \mathcal{A}.$ From these facts, for arbitrary elements $\lambda \in \mathds{C}$ and $a\in \mathcal{A}$ satisfying $a^{*}=-a,$ we have $0=\Phi (a(\lambda 1_{\mathcal{A}})+(\lambda 1_{\mathcal{A}})a^{*})=\Phi (a)\Phi (\lambda 1_{\mathcal{A}})+\Phi (\lambda 1_{\mathcal{A}})\Phi (a)^{*}$ which implies that $\Phi (a)\Phi (\lambda 1_{\mathcal{A}})=\Phi (\lambda 1_{\mathcal{A}})\Phi (a).$ It follows that $b\Phi (\lambda 1_{\mathcal{A}})=\Phi (\lambda 1_{\mathcal{A}})b,$ for any anti-self-adjoint element $b\in \mathcal{B}.$ As consequence, we have $b\Phi (\lambda 1_{\mathcal{A}})=\Phi (\lambda 1_{\mathcal{A}})b,$ for an arbitrary element $b\in \mathcal{B},$ since we can write any element $b\in \mathcal{B}$ in the form $b=b_{1}+ib_{2},$ where $b_{1}=\frac{b-b^{*}}{2}$ and $b_{2}=\frac{b+b^{*}}{2i}$ are anti-self-adjoint elements of $\mathcal{A}.$ This results that $\Phi (\lambda 1_{\mathcal{A}})\in \mathds{C}1_{\mathcal{B}}.$ It follows that $\Phi (\mathds{C}1_{\mathcal{A}})\subseteq \mathds{C}1_{\mathcal{B}}.$ Since $\Phi ^{-1}$ has the same properties of $\Phi ,$ then applying a similar argument we obtain $\mathds{C}1_{\mathcal{B}}\subseteq \Phi (\mathds{C}1_{\mathcal{A}}).$ Consequently, $\Phi (\mathds{C}1_{\mathcal{A}})=\mathds{C}1_{\mathcal{B}}.$
\end{proof}

\begin{lemma}\label{lem26} $\Phi (1_{\mathcal{A}})=1_{\mathcal{B}}$ and there is a real number $\nu $ satisfying $\nu ^{2}=1$ and such that $\Phi (ia)=i\nu \Phi (a),$ for all element $a\in \mathcal{A}.$
\end{lemma}
\begin{proof} By Lemma \ref{lem25}, there is a non-zero element $\lambda \in \mathds{C}$ such that $\Phi (1_{\mathcal{A}})=\lambda 1_{\mathcal{B}}.$ Fix a non-trivial projection $p\in \mathcal{A}.$ It follows that $2\Phi (p)=\Phi (2p)=\Phi (p1_{\mathcal{A}}+1_{\mathcal{A}}p^{*})=\Phi (p)\Phi (1_{\mathcal{A}})+\Phi (1_{\mathcal{A}})\Phi (p)^{*}=\lambda (\Phi (p)+\Phi (p)^{*}).$ This allows us to write $\Phi (p)=\lambda q_{1},$ where $q_{1}=\frac{1}{2}(\Phi (p)+\Phi (p)^{*})$ is an self-adjoint element of $\mathcal{B}.$ Also, $2\Phi (p)=\Phi (2p)=\Phi (1_{\mathcal{A}}p+p 1_{\mathcal{A}}^{*})=\Phi (1_{\mathcal{A}})\Phi (p)+\Phi (p)\Phi (1_{\mathcal{A}})^{*}=(\lambda +\overline{\lambda })\Phi (p)$ which results that $\lambda +\overline{\lambda }=2.$ Yet, $2\Phi (p)=\Phi (2p)=\Phi (pp+pp^{*})=\Phi (p)\Phi (p)+\Phi (p)\Phi (p)^{*}=(\lambda ^{2}+\lambda \overline{\lambda })q_{1}^{2}=2\lambda q_{1}^{2}$ which implies that $q_{1}=q_{1}^{2}.$ This show that $q_{1}$ is a non-trivial projection of $\mathcal{B}.$ Set $q_{2}=1_{\mathcal{B}}-q_{1}.$ Then $\mathcal{B}$ has a Peirce decomposition $\mathcal{B}=\mathcal{B}_{11}\oplus \mathcal{B}_{12}\oplus \mathcal{B}_{21}\oplus \mathcal{B}_{22},$ where $\mathcal{B}_{ij}=q_{i}\mathcal{B}q_{j}$ $(i,j=1,2) ,$ satisfying the following multiplicative relations: (i) $\mathcal{B}_{ij}\mathcal{B}_{jl}\subseteq \mathcal{B}_{il}$ $(i,j,l=1,2),$ (ii) $\mathcal{B}_{ij}\mathcal{B}_{ij}\subseteq \mathcal{B}_{ji}$ $(i,j=1,2),$ (iii) $\mathcal{B}_{ij}\mathcal{B}_{kl}=0$ if $j\neq k$ and $(i,j)\neq (k,l)$ $(i,j,k,l=1,2)$ and (iv) $b_{ij}^{2}=0$ for all elements $b_{ij}\in \mathcal{B}_{ij}$ $(i,j=1,2;i\neq j).$ Let $a$ be an element of $\mathcal{A}$ such that $s=pa(1_{\mathcal{A}}-p)$ is a non-zero element. Then $\Phi (s)=\Phi (ps+sp^{*})=\Phi(p)\Phi (s)+\Phi (s)\Phi (p)^{*}=(\lambda q_{1})\Phi (s)+\overline{\lambda }\Phi (s)q_{1}$ which implies that $q_{1}\Phi(s)q_{1}=0,$ $q_{2}\Phi(s)q_{2}=0,$ $(1-\lambda )q_{1}\Phi(s)q_{2}=0$ and $(1-\overline{\lambda })q_{2}\Phi(s)q_{1}=0.$ If $\lambda \neq 1,$ then  $q_{1}\Phi(s)q_{2}=0$ and $q_{2}\Phi(s)q_{1}=0.$ It follows that $\Phi (s)=0$ from which we get $s=0$ and resulting in a contradiction. Thus we must have $\lambda =1$ which leads to $\Phi (1_{\mathcal{A}})=1_{\mathcal{B}}$ and $\Phi (s)=q_{1}\Phi(s)q_{2}+q_{2}\Phi(s)q_{1}.$ Now, we observe that $0=\Phi (sp+ps^{*})=\Phi (s)\Phi (p)+\Phi (p)\Phi (s)^{*}=\Phi (s)q_{1}+q_{1}\Phi (s)^{*}.$ As consequence we obtain $q_{2}\Phi (s)q_{1}=0$ and $\Phi (s)=q_{1}\Phi(s)q_{2}.$ Let $\nu $ be a non-zero real number such that $\Phi (i1_{\mathcal{A}})=i\nu 1_{\mathcal{B}},$ by Lemma \ref{lem25}. Then $2\Phi (ip)=\Phi (p(i1_{\mathcal{A}})+(i1_{\mathcal{A}})p^{*})=\Phi (p)\Phi(i1_{\mathcal{A}})+\Phi (i1_{\mathcal{A}})\Phi (p)^{*}=2i\nu q_{1}$ which implies that $\Phi (ip)=i\nu q_{1}.$ It follows that $\Phi (is)=\Phi ((ip)s+s(ip)^{*})=\Phi(ip)\Phi (s)+\Phi (s)\Phi (ip)^{*}=(i\nu q_{1})\Phi (s)+\Phi (s)(i\nu q_{1})^{*}=i\nu \Phi (s)$ from which we have that $\Phi (is)=i\nu \Phi (s).$ Moreover, as $-\Phi (s)=\Phi(-s)=\Phi ((ip)(is)+(is)(ip)^{*})=\Phi(ip)\Phi (is)+\Phi (is)\Phi (ip)^{*}=(i\nu q_{1})(i\nu \Phi (s))+(i\nu \Phi (s))(i\nu q_{1})^{*}=-\nu ^{2}\Phi (s)$ we still get that $\nu ^{2}=1.$ Let $t$ be an element of $\mathcal{A}$ such that $t=t^{*}.$ Then $2\Phi (t)=\Phi (2t)=\Phi (t1_{\mathcal{A}}+1_{\mathcal{A}}t^{*})=\Phi (t)\Phi (1_{\mathcal{A}})+\Phi (1_{\mathcal{A}})\Phi (t)^{*}=\Phi (t)+\Phi (t)^{*}$ which implies that $\Phi (t)=\Phi (t)^{*}.$ This results that $2\Phi (it)=\Phi (2it)=\Phi (t(i1_{\mathcal{A}})+(i1_{\mathcal{A}})t^{*})=\Phi (t)\Phi (i1_{\mathcal{A}})+\Phi (i1_{\mathcal{A}})\Phi (t)^{*}=\Phi (t)(i\nu 1_{\mathcal{B}})+(i\nu 1_{\mathcal{B}})\Phi (t)^{*}$ which allows us to conclude that $\Phi (it)=i\nu \Phi (t).$ Therefore, for an arbitrary element $a\in \mathcal{A},$ let us write $a=a_{1}+ia_{2},$ where $a_{1}=\frac{a+a^{*}}{2}$ and $a_{2}=\frac{a-a^{*}}{2i}$ are self-adjoint elements of $\mathcal{A}.$ Then $\Phi (ia)=\Phi (ia_{1}-a_{2})=\Phi (ia_{1})-\Phi (a_{2})=i\nu \Phi (a_{1})+i\nu \Phi (ia_{2})=i\nu \Phi (a_{1}+a_{2})=i\nu \Phi (a).$
\end{proof}

\begin{lemma}\label{lem27} $\Phi :\mathcal{A}\rightarrow \mathcal{B}$ is a $\ast $-ring isomorphism.
\end{lemma}
\begin{proof} By Lemma \ref{lem26}, for arbitrary elements $a,b\in \mathcal{A},$ we have
{\allowdisplaybreaks\begin{multline}\allowdisplaybreaks\label{fundident2}
\Phi (ab-ba^{*})=\Phi ((ia)(-ib)+(-ib)(ia)^{*})=\Phi (ia)\Phi (-ib)\\+\Phi (-ib)\Phi (ia)^{*}
= (i\nu \Phi (a))(-i\nu\Phi (b))+(-i\nu\Phi (b))(i\nu \Phi (a))^{*}\\
= \Phi (a)\Phi (b)-\Phi (b)\Phi (a)^{*}.
\end{multline}}
Hence, from (\ref{fundident}) and (\ref{fundident2}) we get $\Phi (ab)=\Phi (a)\Phi (b)$ and $\Phi (ba^{*})=\Phi (b)\Phi (a)^{*}.$ As result of the last equation, we have $\Phi (a^{*})=\Phi (1_{\mathcal{A}}a^{*})=\Phi (1_{\mathcal{A}})\Phi (a)^{*}=1_{\mathcal{B}}\Phi (a)^{*}=\Phi (a)^{*},$ for all elements $a\in \mathcal{A}.$ Therefore, we can conclude that $\Phi $ is a $\ast$-ring isomorphism.
\end{proof}

The Main Theorem is proved.

\noindent Second case: the product $ab-ba^{*}.$ 

\begin{lemma}\label{lem255} $\Phi (\mathds{R}1_{\mathcal{A}})=\mathds{R}1_{\mathcal{B}},$ the map $\Phi $ preserves self-adjoint elements in both directions and $\Phi (\mathds{C}1_{\mathcal{A}})=\mathds{C}1_{\mathcal{B}}.$
\end{lemma}
\begin{proof} First of all, note that for any elements $a,b\in \mathcal{A}$ holds $ab-ba^{*}=0$ if and only if $\Phi (a)\Phi (b)-\Phi (a)\Phi (b)^{*}=0.$ Hence, for any elements $\alpha \in \mathds{R}$ and $b\in \mathcal{A},$ we have $0=\Phi ((\alpha 1_{\mathcal{A}})b-b(\alpha 1_{\mathcal{A}})^{*})=\Phi (\alpha 1_{\mathcal{A}})\Phi (b)-\Phi (b)\Phi (\alpha 1_{\mathcal{A}})^{*}=0$ from which we get $\Phi (\alpha 1_{\mathcal{A}})\in \mathds{R}1_{\mathcal{B}},$ by Lemma \ref{lem24}. As $\alpha $ is any element, we conclude that $\Phi (\mathds{R}1_{\mathcal{A}})\subseteq \mathds{R}1_{\mathcal{B}}.$ Since $\Phi ^{-1}$ has the same properties of $\Phi ,$ then by a similar argument we prove that $\mathds{R}1_{\mathcal{B}}\subseteq \Phi (\mathds{R}1_{\mathcal{A}}).$ Consequently, we obtain $\Phi (\mathds{R}1_{\mathcal{A}})=\mathds{R}1_{\mathcal{B}}.$

Let $a$ be an element of $\mathcal{A}$ satisfying $a=a^{*}$ (called a {\it sef-adjoint element}). Then $0=\Phi (a1_{\mathcal{A}}-1_{\mathcal{A}}a^{*})=\Phi (a)\Phi (1_{\mathcal{A}})-\Phi (1_{\mathcal{A}})\Phi (a)^{*}$ from which follows that $\Phi (a)=\Phi (a)^{*}.$ Since $\Phi ^{-1}$ has the same properties of $\Phi ,$ then by a similar argument we prove that if $\Phi (a)=\Phi (a)^{*},$ then $a=a^{*},$ for all such elements $a\in \mathcal{A}.$ Consequently, $a=a^{*}$ if and only if $\Phi (a)=\Phi (a)^{*},$ for such elements $a\in \mathcal{A}.$

For any elements $\lambda \in \mathds{C}$ and $a\in \mathcal{A}$ satisfying $a=a^{*},$ we have $0=\Phi (a(\lambda 1_{\mathcal{A}})-(\lambda 1_{\mathcal{A}})a^{*})=\Phi (a)\Phi (\lambda 1_{\mathcal{A}})-\Phi (\lambda 1_{\mathcal{A}})\Phi (a)^{*}$ which implies that $\Phi (a)\Phi (\lambda 1_{\mathcal{A}})-\Phi (\lambda 1_{\mathcal{A}})\Phi (a)=0.$ This results that $b\Phi (\lambda 1_{\mathcal{A}})-\Phi (\lambda 1_{\mathcal{A}})b=0,$ for any sel-adjoint element $b\in \mathcal{B}.$ As a consequence, we obtain $b\Phi (\lambda 1_{\mathcal{A}})-\Phi (\lambda 1_{\mathcal{A}})b=0,$ for any element $b\in \mathcal{B},$ since any element $b\in \mathcal{B}$ can be written as a sum $b=b_{1}+ib_{2},$ where $b_{1}=\frac{b+b^{*}}{2}$ and $b_{2}=\frac{b-b^{*}}{2i},$ are self-adjoint elements of $\mathcal{B}.$ This shows that $\Phi (\lambda 1_{\mathcal{A}})\in \mathds{C}1_{\mathcal{B}}.$ As $\lambda $ is any element we conclude that $\Phi (\mathds{C}1_{\mathcal{A}})\subseteq \mathds{C}1_{\mathcal{B}}.$ Applying a similar argument we obtain $\mathds{C}1_{\mathcal{B}}\subseteq \Phi (\mathds{C}1_{\mathcal{A}}).$ Consequently, $\Phi (\mathds{C}1_{\mathcal{A}})=\mathds{C}1_{\mathcal{B}}.$
\end{proof}

\begin{lemma}\label{lem266} $\Phi (ia)=i\Phi (a),$ for all elements $a\in \mathcal{A},$ or $\Phi (ia)=-i\Phi (a),$ for all elements $a\in \mathcal{A}.$
\end{lemma}
\begin{proof} According to the hypothesis on $\Phi $ and by Lemma \ref{lem255}, we have $\Phi (\pm \frac{1}{2}i1_{\mathcal{A}})\in (\mathds{C}\setminus \mathds{R})1_{\mathcal{B}}$ and $\Phi (\pm \frac{1}{2}1_{\mathcal{A}})\in \mathds{R}1_{\mathcal{B}}.$ As $\pm \frac{1}{2}i1_{\mathcal{A}}=(\pm \frac{1}{2}i1_{\mathcal{A}})(\frac{1}{2}1_{\mathcal{A}})-(\frac{1}{2}1_{\mathcal{A}})(\pm \frac{1}{2}i1_{\mathcal{A}})^{*},$ then
{\allowdisplaybreaks\begin{multline*}\allowdisplaybreaks
\Phi (\pm \frac{1}{2}i1_{\mathcal{A}})=\Phi ((\pm \frac{1}{2}i1_{\mathcal{A}})(\frac{1}{2}1_{\mathcal{A}})-(\frac{1}{2}1_{\mathcal{A}})(\pm \frac{1}{2}i1_{\mathcal{A}})^{*})\\=\Phi (\pm \frac{1}{2}i1_{\mathcal{A}})\Phi (\frac{1}{2}1_{\mathcal{A}})-\Phi (\frac{1}{2}1_{\mathcal{A}})\Phi (\pm \frac{1}{2}i1_{\mathcal{A}})^{*}
\end{multline*}}
from which we get $\Phi (\pm \frac{1}{2}i1_{\mathcal{A}})^{*}=(\Phi (\frac{1}{2}1_{\mathcal{A}})-1_{\mathcal{B}})\Phi (\frac{1}{2}1_{\mathcal{A}})^{-1}\Phi (\pm \frac{1}{2}i1_{\mathcal{A}}).$ This results that $\Phi (\pm \frac{1}{2}i1_{\mathcal{A}})\in \mathds{R}i1_{\mathcal{B}}.$ As consequence we have 
{\allowdisplaybreaks\begin{multline}\allowdisplaybreaks \label{num22}
\Phi (\frac{1}{2}i1_{\mathcal{A}})=\Phi ((-\frac{1}{2}i1_{\mathcal{A}})(-\frac{1}{2}1_{\mathcal{A}})-(-\frac{1}{2}1_{\mathcal{A}})(-\frac{1}{2}i1_{\mathcal{A}})^{*})\\=\Phi (-\frac{1}{2}i1_{\mathcal{A}})\Phi (-\frac{1}{2}1_{\mathcal{A}})-\Phi (-\frac{1}{2}1_{\mathcal{A}})\Phi (-\frac{1}{2}i1_{\mathcal{A}})^{*}\\=2\Phi (-\frac{1}{2}1_{\mathcal{A}})\Phi (-\frac{1}{2}i1_{\mathcal{A}}),
\end{multline}}
since $\Phi (\pm \frac{1}{2}i1_{\mathcal{A}})^{*}=-\Phi (\pm \frac{1}{2}i1_{\mathcal{A}}).$ Also,
{\allowdisplaybreaks\begin{multline}\allowdisplaybreaks \label{num33}
\Phi (-\frac{1}{2}1_{\mathcal{A}})=\Phi ((\frac{1}{2}i1_{\mathcal{A}})(\frac{1}{2}i1_{\mathcal{A}})-(\frac{1}{2}i1_{\mathcal{A}})(\frac{1}{2}i1_{\mathcal{A}})^{*})\\=\Phi (\frac{1}{2}i1_{\mathcal{A}})\Phi (\frac{1}{2}i1_{\mathcal{A}})-\Phi (\frac{1}{2}i1_{\mathcal{A}})\Phi(\frac{1}{2}i1_{\mathcal{A}})^{*}=2\Phi (\frac{1}{2}i1_{\mathcal{A}})^{2}
\end{multline}}

From (\ref{num22}) and (\ref{num33}) we conclude that $\Phi (-\frac{1}{2}1_{\mathcal{A}})=-\frac{1}{2}1_{\mathcal{A}}$ and $\Phi (\frac{1}{2}i1_{\mathcal{A}})=\pm \frac{1}{2}i1_{\mathcal{A}}.$ Thus, for any element $a\in \mathcal{A},$ we have that $\Phi (ia)=\Phi ((\frac{1}{2}i1_{\mathcal{A}})a-a(\frac{1}{2}i1_{\mathcal{A}})^{*})=\Phi (\frac{1}{2}i1_{\mathcal{A}})\Phi (a)-\Phi (a)\Phi (\frac{1}{2}i1_{\mathcal{A}})^{*}=\big(\Phi (\frac{1}{2}i1_{\mathcal{A}})-\Phi (\frac{1}{2}i1_{\mathcal{A}})^{*}\big)\Phi (a)=2\Phi (\frac{1}{2}i1_{\mathcal{A}})\Phi (a).$ Consequently, $\Phi (ia)=i\Phi (a),$ for all elements $a\in \mathcal{A},$ or $\Phi (ia)=-i\Phi (a),$ for all elements $a\in \mathcal{A}.$
\end{proof}

\begin{lemma} $\Phi :\mathcal{A}\rightarrow \mathcal{B}$ is a $\ast $-ring isomorphism.
\end{lemma}
\begin{proof} Two cases are considered. First case: $\Phi (ia)=i\Phi (a),$ for all elements $a\in \mathcal{A}.$ For any elements $a,b\in \mathcal{A},$ we have
{\allowdisplaybreaks\begin{multline}\allowdisplaybreaks\label{fundident22}
\Phi (ab+ba^{*})=\Phi ((ia)(-ib)-(-ib)(ia)^{*})=\Phi (ia)\Phi (-ib)\\-\Phi (-ib)\Phi (ia)^{*}
=(i\Phi (a))(-i\Phi (b))-(-i\Phi (b))(i\Phi (a))^{*}\\
= \Phi (a)\Phi (b)+\Phi (b)\Phi (a)^{*}.
\end{multline}}
From (\ref{fundident}) and (\ref{fundident22}) we get $\Phi (ab)=\Phi (a)\Phi (b)$ and $\Phi (ba^{*})=\Phi (b)\Phi (a)^{*}.$ It follows that $\Phi (1_{\mathcal{A}})=1_{\mathcal{B}}$ and $\Phi (a^{*})=\Phi (1_{\mathcal{A}}a^{*})=\Phi (1_{\mathcal{A}})\Phi (a)^{*}=1_{\mathcal{B}}\Phi (a)^{*}=\Phi (a)^{*},$ for any element $a\in \mathcal{A}.$ This allows us to conclude that $\Phi $ is a $\ast$-ring isomorphism. Second case: $\Phi (ia)=-i\Phi (a),$ for all elements $a\in \mathcal{A}.$ The proof is entirely similar as in the first case.
\end{proof}

The Main Theorem is proved.



\begin{thebibliography}{99}
\bibitem{Bai} Z. Bai, S. Du, Maps preserving products $XY-YX^{*}$ on von Neumann algebras, J. Math. Anal. Appl. 386 (2012) 103-109.
\bibitem{Cabrera1} M. Cabrera, A. Rodríguez, Non-Associative Normed Algebras, vol. 1: TheVidav-Palmer and Gelfand-Naimark Theorems, Encyclopedia Math. Appl., vol. 154, Cambridge University Press, 2014.
\bibitem{Cabrera2} M. Cabrera, A. Rodríguez, Non-Associative Normed Algebras, vol. 2: Representation Theory and the Zelmanov Approach, Encyclopedia Math. Appl., vol. 167, Cambridge University Press, 2018.
\bibitem{Cui} J. Cui, C.K. Li, Maps preserving product $XY-YX^{*}$ on factor von Neumann algebras, Linear Algebra Appl. 431 (2009) 833-842.
\bibitem{BrunGuz} B. L. M., Ferreira, H. Guzzo Jr., Lie maps on alternative rings. Boll. Unione Mat. Ital. 13 (2020), 181–192.
\bibitem{FerMar} J. C. M. Ferreira, M. G. B. Marietto, Mappings preserving sum of products $a\circ  b-ba^{*}$ on factor von Neumann algebras, Bull. Iran. Math. Soc. (2020) (10 pages).
\bibitem{FerGuz1} J. C. M. Ferreira, H. Guzzo Jr., Jordan elementary maps on alternative rings, Comm. Algebra 42 (2014), 779-794.
\bibitem{FerGuz2} J. C. M. Ferreira, H. Guzzo Jr., Multiplicative mappings of alternative rings, Algebras Groups Geom. 31 (2014), 239-250.
\bibitem{Ji} P. Ji, Z. Liu, Additivity of Jordan maps on standard Jordan operator algebras, Linear Algebra Appl. 430 (2009) 335-343.
\bibitem{Ji2} P. Ji, Additivity of Jordan maps on Jordan algebras, Linear Algebra Appl. 431 (2009), 179-188.
\bibitem{Li} C. Li, F. Lu, X. Fang, Nonlinear mappings preserving product $XY+YX^{*}$ on factor von Neumann algebras, Linear Algebra Appl. 438 (2013) 2339-2345.
\bibitem{Lu1} F. Lu, Additivity of Jordan maps on standard operator algebras, Linear Algebra Appl. 357 (2002) 123-131.
\bibitem{Lu2} F. Lu, Jordan maps on associative algebras, Comm. Algebra 31 (2003) 2273-2286.
\bibitem{Martindale} W.S. Martindale III, When are multiplicative mappings additive? Proc. Amer. Math. Soc. 21 (1969) 695-698.
\bibitem{Martindale1} W. S. Martindale III, Lie isomorphisms of simple rings, J. London Math. Soc. 44 (1969) 213–221.
\bibitem{Miers1} C. Miers, Lie isomorphisms of operator algebras, Pacific J. Math. 38 (1971) 717-735.  
\bibitem{Miers2} C. Miers, Lie isomorphisms of factors, Trans. Amer. Math. Soc. 147 (1970) 55-63.
\bibitem{Zhang} J. H. Zhang and F. J. Zhang, Nonlinear maps preserving Lie products on factor von Neumann algebras, Linear Algebra Appl. 429 (2008) 18-30.
\end{thebibliography}
\end{document}